\documentclass{article}

\usepackage[utf8]{inputenc}
\usepackage[english]{babel}
\usepackage{amsmath, amsfonts, amssymb, amsthm}
\usepackage{graphicx}
\usepackage{algorithm}
\usepackage{algorithmic}
\usepackage{enumerate}
\usepackage[usenames,dvipsnames]{xcolor}
\usepackage[left=1.5in,right=1.5in,top=2.5cm,bottom=2.5cm]{geometry}
\usepackage{wrapfig}
\usepackage{caption}
\usepackage{subfig}

\usepackage{hyperref}
\usepackage[]{natbib}

\newcommand{\Rbb}{\mathbb{R}}

\newcommand{\Nbb}{\mathbb{N}}

\newcommand{\Fcal}{\mathcal{F}}

\newcommand{\Ical}{\mathcal{I}}

\newcommand{\Mcal}{\mathcal{M}}

\newcommand{\Ocal}{\mathcal{O}}

\newcommand{\Vcal}{\mathcal{V}}

\newcommand{\Xcal}{\mathcal{X}}

\newcommand{\Zcal}{\mathcal{Z}}

\DeclareMathOperator*{\argmax}{arg\,max}
\DeclareMathOperator*{\argmin}{arg\,min}

\newcommand \al[1]{\begin{align*}
#1
\end{align*}
}

\newcommand \eqv{\Leftrightarrow}

\newcommand \bra{\left\langle}
\newcommand \ket{\right\rangle}
\newcommand \braket[2]{\bra #1, #2 \ket}

\newcommand \parenth[1]{\left( #1 \right)}

\usepackage{listings} 
\usepackage{framed}
\newtheorem{definition}{Definition}[section]
\newtheorem*{definition*}{Definition}
\newtheorem{lemma}{Lemma}
\newtheorem*{lemma*}{Lemma}
\newtheorem{theorem}{Theorem}

\newtheorem{remark}{Remark}

\newtheorem*{fact*}{Fact}
\newtheorem{proposition}{Proposition}
\newtheorem*{proposition*}{Proposition}
\newtheorem{corollary}{Corollary}
\newtheorem{assumption}{Assumption}

\DeclareMathOperator{\MD}{MD}
\DeclareMathOperator{\AMD}{AMD}
\DeclareMathOperator{\SMD}{SMD}
\DeclareMathOperator{\SAMD}{SAMD}
\DeclareMathOperator{\tr}{tr}
\newcommand\Ebb{\mathbb{E}}

\title{Acceleration and Averaging\\In Stochastic Mirror Descent Dynamics}

\author{
Walid Krichene\thanks{Research and Machine Intelligence, Google Inc. \texttt{walidk@google.com}}
\and
Peter Bartlett\thanks{Department of Statistics and Computer Science division, University of California at Berkeley, and Queensland University of Technology. \texttt{bartlett@cs.berkeley.edu}
}}

\begin{document}

\maketitle

\begin{abstract}
We formulate and study a general family of (continuous-time) stochastic dynamics for accelerated first-order minimization of smooth convex functions.

Building on an averaging formulation of accelerated mirror descent, we propose a stochastic variant in which the gradient is contaminated by noise, and study the resulting stochastic differential equation. We prove a bound on the rate of change of an energy function associated with the problem, then use it to derive estimates of convergence rates of the function values, (a.s. and in expectation) both for persistent and asymptotically vanishing noise. We discuss the interaction between the parameters of the dynamics (learning rate and averaging weights) and the covariation of the noise process, and show, in particular, how the asymptotic rate of covariation affects the choice of parameters and, ultimately, the convergence rate.
\end{abstract}

\section{Introduction}
\label{sec:intro}
We consider the constrained convex minimization problem
\[
\min_{x \in \Xcal} f(x),
\]
where $\Xcal$ is a closed, convex, subset of $E = \Rbb^n$, and $f$ is a proper closed convex function, assumed to be differentiable with Lipschitz gradient, and we will denote $\Xcal^\star$ the set of minimizers of this problem, assumed to be non-empty. First-order optimization methods play an important role in minimizing such functions, in particular in large-scale machine learning applications, in which the dimensionality (number of features) and size (number of samples) in typical datasets makes higher-order methods intractable. 
Many such algorithms can be viewed as a discretization of a continuous-time dynamics. The simplest example is gradient descent, which can be viewed as the discretization of the gradient flow dynamics $\dot x(t) = -\nabla f(x(t))$, where $\dot x(t)$ denotes the time derivative of a $C^1$ trajectory $x(t)$. An important generalization of gradient descent, which elegantly handles the constraint set $\Xcal$, was developed by~\cite{nemirovski1983problem}, and termed mirror descent: it couples a dual variable $z(t)$ accumulating gradients, and its ``mirror'' primal variable~$x(t)$. More specifically, the dynamics are given by
\begin{equation}
\label{eq:ODE_MD}
\MD\begin{cases}
\dot z(t) = -\nabla f(x(t)) \\
x(t) = \nabla \psi^*(z(t)),
\end{cases}
\end{equation}
where $\nabla \psi^*: E^* \to \Xcal$ is a Lipschitz function defined on the entire dual space $E^*$, with values in the feasible set $\Xcal$; it is often referred to as a mirror map, and we will recall its definition and properties in Section~\ref{sec:smd}. Mirror descent can be viewed as a generalization of projected gradient descent, where the Euclidean projection is replaced by the mirror map $\nabla \psi^*$~\citep{beck2003mirror}. This makes it possible to adapt the choice of the mirror map to the particular geometry of the problem, leading to closed-form solutions of the projection, or to better dependence on the dimension $n$, see~\citep{bental2001Lectures}, \citep{bental2001oredered}.

\subsection*{Continuous-time dynamics}
Although optimization methods are inherently discrete, the continuous-time point of view can help in their design and analysis, since it can leverage the rich literature on dynamical systems, control theory, and mechanics, see~\citep{helmke1994optimization}, \citep{bloch1994hamiltonian}, and the references therein. Continuous-time models are also commonly used in financial applications, such as option pricing~\citep{black1973pricing}, even though the actions are taken in discrete time. In convex optimization, beyond simplifying the analysis, continuous-time models have also motivated new algorithms and heuristics: mirror descent is one such example, since it was originally motivated in continuous-time (Chapter 3 in~\citep{nemirovski1983problem}). In a more recent line of work (\citep{su2014differential}, \citep{krichene2015amd}, \citep{wibisono2016variational}), Nesterov's accelerated method~\citep{nesterov1983method} was shown to be the discretization of a second-order ordinary differential equation (ODE), which, in the unconstrained case, can be interpreted as a damped non-linear oscillator~\citep{cabot2009dissipation, attouch2015fast}. This motivated a restarting heuristic \citep{ODonoghue2015adaptive}, which aims at further dissipating the energy of the oscillator. \cite{krichene2015amd} generalized this ODE to mirror descent dynamics, and gave an averaging interpretation (a connection which was previously pointed out by~\cite{flammarion2015colt} for quadratic functions). This averaging formulation is the starting point of this paper, in which we introduce and study a stochastic variant of accelerated mirror descent dynamics.

\subsection*{Stochastic dynamics and related work}
The dynamics which we discussed so far (gradient descent, mirror descent, and their accelerated variants) are \emph{deterministic} first-order dynamics, since they use the exact gradient $\nabla f$. However, in many machine learning applications, evaluating the exact gradient $\nabla f$ can be prohibitively expensive, e.g. when the objective function $f$ involves the sum of loss functions over a training set, of the form $f(x) = \frac{1}{|\Ical|}\sum_{i \in \Ical} f_i(x) + g(x)$, where $\Ical$ indexes the training samples, and $g$ is a regularization function. Instead of computing the exact gradient $\nabla f(x) = \frac{1}{|\Ical|}\sum_{i \in \Ical} \nabla f_i(x) + \nabla g(x)$, a common approach is to compute an unbiased, stochastic estimate of the gradient, given by $\frac{1}{|\tilde\Ical|} \sum_{i \in \tilde \Ical} \nabla f_i(x) + \nabla g(x)$, where $\tilde I$ is a uniformly random subset of $\Ical$, indexing a random batch of samples from the training set. This approach motivates the study of stochastic dynamics for convex optimization. But despite an extensive literature on stochastic gradient and mirror descent in discrete time, e.g. \citep{nemirovski2009robust}, \citep{duchi2012ergodic}, \citep{lan2012optimal}, \citep{johnson2013SVRG},  \citep{xiao2014progressive}, and many others, few results are known for stochastic mirror descent in continuous-time. To the best of our knowledge, the only published results are by \cite{raginsky2012continuous} and \cite{mertikopoulos2016convergence}.

In its simplest form, the stochastic gradient flow dynamics can be described by the It\^o stochastic differential equation (SDE) \citep{oksendal2003stochastic}
\[
dX(t) = -\nabla f(X(t)) + \sigma dB(t),
\]
where $B(t)$ denotes a standard Wiener process (Brownian motion). It is well known that this dynamics admits a unique invariant measure with density proportional to the Gibbs distribution $e^{-\frac{2f(x)}{\sigma}}$. Such dynamics have recently played an important role in the analysis of sampling methods~\citep{dalalyan2017sampling}, \citep{bubeck2015langevin}, \citep{cheng2017convergence}, \citep{cheng2017underdamped}, where $f$ is taken to be the logarithm of a target distribution $p$. The stationary distribution of the SDE has also been recently interpreted as an approximate Bayesian inference~\citep{mandt2017SGD}, and to derive convergence rates (in expectation) for smooth, non-convex optimization where the objective is dissipative~\citep{raginsky2017nonconvex}.

For mirror descent dynamics, \cite{raginsky2012continuous} were the first to propose a stochastic variant of the mirror descent ODE~\eqref{eq:ODE_MD}, given by the It\^o SDE:
\begin{equation}
\label{eq:ODE_SMD}
\SMD \begin{cases}
dZ(t) = -\nabla f(X(t)) + \sigma dB(t) \\
X(t) = \nabla \psi^*(Z(t))
\end{cases}
\end{equation}
where $\sigma$ is a constant volatility. In particular, they argued that the function values $f(X(t))$ along sample trajectories do not converge to the minimum value of $f$ due to the persistent noise, but the optimality gap is bounded by a quantity proportional to $\sigma^2$. They also proposed a method to reduce the variance by simultaneously sampling multiple trajectories and linearly coupling them. \cite{mertikopoulos2016convergence} extended the analysis in some important directions: they replaced the constant volatility $\sigma$ with a general volatility matrix $\sigma(x, t)$ which can depend on the current point $x$ and current time $t$, and studied two regimes: the vanishing noise regime given by the condition $\sup_{x \in \Xcal} \sigma(x, t) = o(1/\sqrt{ \log t})$, in which case they prove almost sure convergence of solution trajectories; and the persistent noise regime, given by the condition $\sup_{x \in \Xcal} \sigma(x, t) \leq \sigma_*$ uniformly in $t$, in which case they define a rectified variant of $\SMD$, obtained by replacing the second equation by $X(t) = \nabla \psi^*(Z(t) / s(t))$, where $1/s(t)$ is a sensitivity parameter. The resulting dynamics is given by
\begin{equation}
\label{eq:SMD}
\SMD_s
\begin{cases}
\dot Z(t) = -\nabla f(X(t)) + \sigma(X(t), t) dB(t), \\
X(t) = \nabla \psi^*(Z(t)/s(t)).
\end{cases}
\end{equation}

Intuitively, a decreasing sensitivity reduces the impact of accumulated noise on the primal trajectory. In particular, they prove that with $1/s(t) = 1/\sqrt t$, the function values converge to the optimal value at a $\Ocal(\sqrt{ \log \log t / t})$ rate, almost surely. They also give concentration estimates around interior solutions in the strongly convex case. While these recent results paint a broad picture of mirror descent dynamics under different noise regimes, they leave many questions open: in particular, they do not provide estimates for convergence rates in the vanishing noise regime, which is an important regime in machine learning applications, since one can often control the variance of the gradient estimate, for example by gradually increasing the batch size, as done by~\cite{xiao2014progressive}. Besides, they do not study accelerated dynamics, and the interaction between acceleration and noise remains unexplored in continuous time.

\subsection*{Our contributions}
In this paper, we answer many of the questions left open in previous works. We formulate and study a family of stochastic accelerated mirror descent dynamics, and we characterize the interaction between its different parameters: the volatility of the noise, the (primal and dual) learning rates, and the sensitivity of the mirror map. More specifically:
\begin{itemize}
\item In Theorem~\ref{thm:as}, we give sufficient conditions for a.s. convergence of solution trajectories to the set of minimizers $\Xcal^\star$. In particular, we show that it is possible to guarantee almost sure convergence even when the volatility is unbounded asymptotically.
\item In Theorem~\ref{thm:rate_exp}, we derive a bound on the expected function values.
\item In Theorem~\ref{thm:rate_as}, we provide estimates of sample trajectory convergence rates.
\end{itemize}

The rest of the paper is organized as follows: We start by reviewing the building blocks of our construction in Section~\ref{sec:smd}, then formulate the stochastic dynamics in Section~\ref{sec:sto}, and prove two instrumental lemmas. Section~\ref{sec:convergence} is dedicated to the convergence results. We give additional numerical examples in Section~\ref{sec:numerics}, to illustrate the effect of acceleration on stochastic mirror descent dynamics. We conclude with a brief discussion in Section~\ref{sec:discussion}.

\section{Accelerated Mirror Descent Dynamics}
\label{sec:smd}

\subsection{Smooth mirror maps}
The mirror map is central in defining mirror descent dynamics. We first give a generic method for constructing mirror maps, adapted to the feasible set $\Xcal \subset E$. We fix a pair of dual reference norms, $\|\cdot\|, \|\cdot\|_*$, defined, respectively, on $E$ and its dual space $E^*$. We say that a map $F: E \to E^*$ is Lipschitz continuous on $\Xcal$ with constant $L$ if for all $x, x' \in \Xcal$, $\|F(x) - F(x')\|_* \leq L \|x - x'\|$. We recall that the effective domain of a convex function $\psi$ is the set $\{x \in E: \psi(x) < \infty\}$, and its convex conjugate $\psi^* : E^* \to \Rbb$ is defined on $E^*$ by $\psi^*(z) = \sup_{x \in \Xcal} \braket{z}{x} - \psi(x)$. We recall that the sub-differential of $\psi$ at $x$ is the set $\partial \psi(x) = \{g \in E^* : \psi(x') \geq \psi(x) + \braket{g}{x'-x} \ \forall x' \in \Xcal\}$, and that $\psi$ is said to be $\mu$-strongly convex (w.r.t. $\|\cdot\|$) if $\forall x, x' \in \Xcal, \ \forall g \in \partial \psi(x)$, $\psi(x) \geq \psi(x') + \braket{g}{x'-x} + \frac{\mu}{2}\|x' - x\|^2$.

\begin{proposition}
\label{prop:mirror_map}
Let $\psi$ be a $\mu$-strongly convex function (w.r.t. $\|\cdot\|$) with effective domain $\Xcal$, and let $\psi^*$ be its convex conjugate. Then $\psi^*$ is finite and differentiable on all of $E^*$, $\nabla \psi^*$ is $\frac{1}{\mu}$-Lipschitz, and has values in $\Xcal$: specifically, for all $z \in E^*$,
\begin{equation}
\label{eq:mirror_solution}
\nabla \psi^*(z) = \argmax_{x \in \Xcal} \braket{z}{x} - \psi(x).
\end{equation}
\end{proposition}
This follows from standard results from convex analysis, e.g. Theorems 13.3 and 25.3 in~\citep{rockafellar1970convex}. To give an example of a Lipschitz mirror map, take $\psi$ to be the squared Euclidean norm, $\psi(x) = \frac{1}{2} \|x\|_2^2$. Then $\psi^*(z) = \argmax_{x \in \Xcal} \braket{z}{x} - \frac{1}{2}\|x\|^2_2 = \argmin_{x \in \Xcal} \|z - x\|_2^2$, and the mirror map reduces to the Euclidean projection on $\Xcal$. It is worth noting that although one can theoretically construct a smooth mirror map given any convex feasible set $\Xcal$, using Proposition~\ref{prop:mirror_map}, this does not necessarily mean that the mirror map can be implemented efficiently, since in its general form, it is given by the solution to the problem~\eqref{eq:mirror_solution}. However, many convex sets have known mirror maps that are efficient to compute. For a concrete example, when $\Xcal$ is the probability simplex $\Delta = \{x \in \Rbb^n_+ : \sum_{i = 1}^n x_i = 1\}$, choosing $\psi$ to be the negative entropy $\psi(x) = \sum_{i = 1}^n x_i \log x_i$ yields a closed-form mirror map given by $(\nabla \psi^*(z))_i = e^{z_i} / \sum_{j = 1}^n e^{z_j}$, see e.g.~\cite{banerjee2005clustering} for additional examples. We will make the following regularity assumption throughout the paper:
\begin{assumption}
\label{assumption:lip}
$\Xcal$ is convex and closed, $\Xcal^\star$ is non-empty, $\psi$ is strongly convex continuous and non-negative on~$\Xcal$, $\psi^*$ is twice differentiable with a Lipschitz gradient, and $f$ is differentiable with Lipschitz gradient. We denote by $L_{\psi^*}$ the Lipschitz constant of $\nabla \psi^*$, and by $L_f$ the Lipschitz constant of $\nabla f$.
\end{assumption}

The assumption that $\psi$ is non-negative is made without loss of generality: since $\psi$ is strongly convex, its infimum is finite, and one can simply translate $\psi$ (without changing the mirror map). Some of our results will also require the feasible set to be compact, so we formulate the following assumption:
\begin{assumption}
\label{assumption:compact}
The feasible set $\Xcal$ is compact.
\end{assumption}
\subsection{Averaging formulation of accelerated mirror descent}
We start from the averaging formulation of accelerated mirror descent given by~\cite{krichene2015amd}, and propose a variant which includes a time-varying sensitivity parameter, similar to~\cite{mertikopoulos2016convergence}. 
We consider the following ODE:
\begin{equation}
\AMD_{\eta, a, s}
\begin{cases}
\dot z(t) = -\eta(t) \nabla f(x(t)) \\
\dot x(t) = a(t) (\nabla \psi^*(z(t) / s(t)) - x(t)),
\end{cases}
\end{equation}
with initial conditions $(x(t_0), z(t_0)) = (x_0, z_0)$. The ODE system is parameterized by the following functions, all assumed to be positive and continuous on $[t_0, \infty)$.
\begin{itemize}
\item $s(t)$ is a non-decreasing, inverse sensitivity parameter. As we will see, the main role of $s(t)$ will be to scale the noise term, in order to reduce its impact on the primal trajectory.
\item $\eta(t)$ is a learning rate in the dual space.
\item $a(t)$ is an averaging rate in the primal space. In order to see this connection with averaging, we can rewrite the primal ODE in integral form as a weighted average of the mirror trajectory $\Mcal_t = \{\nabla \psi^*(z(\tau)/s(\tau)), \tau \in [t_0, t]\}$ as follows: let $a(t) = \frac{\dot w(t)}{w(t)}$ (equivalently, $w(t) = w(t_0) e^{\int_{t_0}^t a(\tau) d\tau}$), then the primal ODE is equivalent to
\[
w(t) \dot x(t) + \dot w(t) x(t) = \dot w(t)\nabla \psi^*(z(t) / s(t)),
\]
and integrating and rearranging,
\begin{equation}
\label{eq:primal_integral_form}
x(t) = \frac{x(t_0)w(t_0) + \int_{t_0}^t \dot w(\tau) \nabla \psi^*(Z(\tau) / s(\tau)) d\tau}{w(t)}.
\end{equation}
\end{itemize}

\subsection{Energy decay}

Next, we define an energy function which will be central in our analysis. The analysis of continuous-time dynamics often relies on a Lyapunov argument (in reference to~\cite{lyapunov1992general}): one starts by defining a non-negative energy function, then bounding its rate of change along solution trajectories. This bound can then be used to prove convergence to the set of minimizers $\Xcal^\star$. We will consider a modified version of the energy function used by~\cite{krichene2016adaptive}: given a positive, $C^1$ function $r(t)$, and a pair of optimal primal-dual points $(x^\star, z^\star)$ such that\footnote{Note that in general, $\nabla \psi^*$ may not be surjective (specifically, points on the boundary of $\Xcal$ may not be attained), but the analysis can be extended to such cases by replacing the Bregman divergence term in $L$ by the Fenchel coupling defined by~\cite{mertikopoulos2016convergence}.} $x^\star \in \Xcal^\star$ and $\nabla \psi^*(z^\star) = x^\star$, let
\begin{equation}
\label{eq:lyap}
L(x, z, t) = r(t)(f(x) - f(x^\star)) + s(t) D_{\psi^\star}(z(t) / s(t), z^\star).
\end{equation}
Here, $D_{\psi^*}$ is the Bregman divergence~\citep{bregman1967relaxation} associated to $\psi^*$, defined by
\al{
D_{\psi^*}(z', z) = \psi^*(z') - \psi^*(z) - \braket{\nabla \psi^*(z)}{z'-z}, && \text{for all $z, z' \in E^*$}.
}
Then we can prove a bound on the time derivative of $L$ along solution trajectories of $\AMD_{\eta, a, s}$, given in the following proposition. To keep the equations compact, we will occasionally omit explicit dependence on time, and write, e.g. $z / s$ instead of $z(t) / s(t)$.
\begin{lemma}
\label{lem:lyap_bound}
Suppose that Assumption~\ref{assumption:lip} holds, and suppose that $a = \eta / r$. Then under $\AMD_{\eta, \eta/r, s}$, for all $t \geq t_0$,
\begin{equation}
\label{eq:lyap_bound}
\frac{d}{dt} L(x(t), z(t), t) \leq (f(x(t)) - f(x^\star))(\dot r(t) - \eta(t)) + \psi(x^\star) \dot s(t)
\end{equation}
\end{lemma}
\begin{proof}
We start by recalling a Bregman identity which will be useful in the proof. We have for all $x \in E$ and $z \in E^*$, $\psi(x) + \psi^*(z) = \braket{x}{z} \eqv x \in \partial \psi^*(z) \eqv z \in \partial \psi(x)$ (Theorem 23.5 in \cite{rockafellar1970convex}). Thus
\begin{align}
\psi(\nabla \psi^*(z_1)) - \psi(\nabla \psi^*(z_2))
&= \braket{\nabla \psi^*(z_1)}{z_1} - \psi^*(z_1) - \braket{\nabla \psi^*(z_2)}{z_2} + \psi^*(z_2) \notag \\
&= D_{\psi^*}(z_2, z_1) - \braket{\nabla \psi^*(z_2) - \nabla \psi^*(z_1)}{z_2}. \label{eq:Bregman_identity}
\end{align}

We proceed by bounding the rate of change of the Bregman divergence term:
\begin{align}
\frac{d}{dt} s(t) &D_{\psi^*}(z(t)/s(t), z^\star) \notag\\
&= \dot s D_{\psi^*}(z/s, z^\star) + s \braket{\nabla \psi^*(z/s) - \nabla \psi^*(z^\star)}{\dot z/s - \dot s z/s^2} \notag\\
&= \braket{\nabla \psi^*(z/s) - x^\star}{\dot z} + \dot s (D_{\psi^*}(z/s, z^\star) - \braket{\nabla \psi^*(z/s) - \nabla \psi^*(z^\star)}{z/s}) \notag\\
&= \braket{\nabla \psi^*(z/s) - x^\star}{\dot z} + \dot s (\psi(x^\star) - \psi(\nabla \psi^*(z/s))) \notag\\
&\leq \braket{\nabla \psi^*(z/s) - x^\star}{\dot z} + \dot s \psi(x^\star), \label{eq:Bregman_change_bound}
\end{align}
where the third equality uses the identity~\eqref{eq:Bregman_identity}, and the last inequality follows from the assumption that $s$ is non-decreasing, and that $\psi$ is non-negative. Using this expression, we can then compute
\al{
\frac{d}{dt} L(x(t), z(t), t)
&\leq \dot r (f(x) - f(x^\star)) + r \braket{\nabla f(x)}{\dot x} + \braket{\nabla \psi^*(z / s) - x^\star}{\dot z} + \psi(x^\star) \dot s \\
&= \dot r (f(x) - f(x^\star)) + r \braket{\nabla f(x)}{\dot x} + \braket{\dot x/a + x - x^\star}{-\eta \nabla f(x)} + \psi(x^\star) \dot s \\
&\leq (f(x) - f(x^\star))(\dot r - \eta) + \braket{\nabla f(x)}{\dot x} (r - \eta/a) + \psi(x^\star) \dot s,
}
where we used the chain rule in the first equality, we plugged in the expression of $\dot z$ and $\nabla \psi^*(z/s)$ from $\AMD_{\eta, a, s}$ to obtain the second equality, and used convexity of $f$ in the last inequality. The assumption $a = \eta / r$ ensures that the middle term vanishes\footnote{Note that this assumption can be replaced by an adaptive rate $a(t)$ similar to the heuristic developed in~\citep{krichene2016adaptive}.}, which concludes the proof.
\end{proof}

As a consequence of the previous proposition, we can prove the following convergence rate:
\begin{corollary}
\label{cor:deterministic_rate}
Suppose that $a = \eta / r$ and that $\eta \geq \dot r$. Then under $\AMD_{\eta, \eta/r, s}$, for all $t \geq t_0$
\[
f(x(t)) - f(x^\star) \leq \frac{\psi(x^\star) (s(t) - s(t_0))+ L(x_0, z_0, t_0)}{r(t)}.
\]
\end{corollary}
\begin{proof}
Starting from the bound~\eqref{eq:lyap_bound}, the first term is non-positive by assumption on $\eta$. Integrating, we have
\[
L(x(t), z(t), t) - L(x_0, z_0, t_0) \leq \psi(x^\star) (s(t) - s(t_0)),
\]
and we conclude by observing that by definition of $L$,
\al{
f(x(t) - f(x^\star))
&= \frac{L(x(t), z(t), t) - s(t)D_{\psi^*}(z(t)/s(t), z^\star)}{r(t)} \\
&\leq \frac{L(x(t), z(t), t)}{r(t)}
}
since the Bregman divergence term is non-negative).
\end{proof}

\begin{remark}
In the deterministic case, it appears from the bound of Corollary~\ref{cor:deterministic_rate} that a strictly increasing $s(t)$ would degrade the convergence rate. However, as we will see in the next section, this parameter will be essential in controlling the effect of noise.
\end{remark}

\begin{remark}
We can recover the (non-accelerated) $\MD$ dynamics as a limiting case of $\AMD$: writing the second equation of $\AMD$ as
\[
x(t) = \nabla {\psi^*}(z(t)/s(t)) - \frac{\dot x(t)}{a(t)}
\]
we can see that $\MD$ can be formally recovered by taking $a(t)$ infinite.

\end{remark}

\begin{remark}
Nesterov's accelerated method can be seen as a special case of $\AMD$ dynamics in the unconstrained Euclidean case, with a quadratic $r(t)$. This is discussed in Appendix~\ref{sec:app:nesterov_ode}.
\end{remark}
\section{Stochastic dynamics}
\label{sec:sto}
We now formulate the stochastic variant of accelerated mirror descent dynamics ($\SAMD$). Intuitively, we would like to replace the gradient terms $\nabla f(x(t))$ in $\AMD_{\eta, a, s}$ by a noisy gradient $G$. More formally, we define a process $G(t)$ which satisfies the It\^o SDE
\begin{equation}
\label{eq:grad_process}
dG(t) = \nabla f(X(t))dt + \sigma(X(t), t) dB(t).
\end{equation}
Here, $B(t) \in \Rbb^n$ is a standard Wiener process with respect to a given filtered probability space $(\Omega, \Fcal, \{\Fcal_t\}_{t\geq t_0}, \mathbb P)$, and $\sigma: (x, t) \mapsto \sigma(x, t) \in \Rbb^{n \times n}$ is a volatility matrix which satisfies the following assumptions:
\begin{assumption}
\label{assumption:volatility}
The volatility matrix $\sigma(x, t)$ is measurable, Lipschitz in $x$ (uniformly in $t$), and continuous in $t$ for all $x$.
\end{assumption}

The resulting stochastic dynamics are given by the SDE system:
\begin{equation}
\SAMD_{\eta, a, s}
\begin{cases}
dZ(t) = - \eta(t) dG(t) = -\eta(t)[\nabla f(X(t))dt + \sigma(X(t), t) dB(t)] \\
dX(t) = a(t) [\nabla \psi^*(Z(t)/s(t)) - X(t)]dt,
\end{cases}
\end{equation}
with initial condition $(X(t_0), Z(t_0)) = (x_0, z_0)$ (we consider deterministic initial conditions for simplicity). The drift term in $\SAMD_{\eta, a, s}$ is identical to the deterministic case, and the volatility term $-\eta(t)\sigma(X(t), t)dB(t)$ represents the noise in the gradient. In particular, we note that the volatility is proportional to $\eta(t)$, to capture the fact that the gradient noise is scaled by the learning rate $\eta$. This formulation is fairly general, and does not assume, in particular, that the different components of the gradient noise are independent, as we can see in the quadratic covariation of the gradient process $G(t)$:
\begin{equation}
\label{eq:quad_cov_G}
d[G_i(t), G_j(t)] = (\sigma(X(t), t) \sigma(X(t), t)^T)_{i, j} dt = \Sigma_{ij}(X(t), t) dt,
\end{equation}
where we defined the infinitesimal covariance matrix $\Sigma(x, t) = \sigma(x, t) \sigma(x, t)^T \in \Rbb^{n \times n}$. Similarly, we have
\begin{equation}
\label{eq:quad_cov}
d[Z_i(t), Z_j(t)] = \eta(t)^2 \Sigma_{ij}(X(t), t) dt.
\end{equation}
We will denote
\begin{equation}
\label{eq:sigma_star}
\sigma_*^2(t) = \sup_{x \in \Xcal} \|\Sigma(x, t)\|_i,
\end{equation}
where $\|\Sigma\|_i = \sup_{\|z\|_* \leq 1}\|\Sigma z\|$ is the induced matrix norm. Contrary to the work of~\citep{raginsky2012continuous,mertikopoulos2016convergence}, we do not assume, a priori, that $\sigma_*(t)$ is uniformly bounded in $t$. Observe that when Assumption~\ref{assumption:compact} holds (i.e. $\Xcal$ is compact), since $\Sigma(x, t)$ is Lipschitz in $x$ and continuous in $t$, $\sigma_*(t)$ is finite for all $t$, and continuous.

\subsection{Illustration of SAMD dynamics}
\label{sec:example}

\begin{figure}[h]
\centering
\includegraphics[width=1.\textwidth]{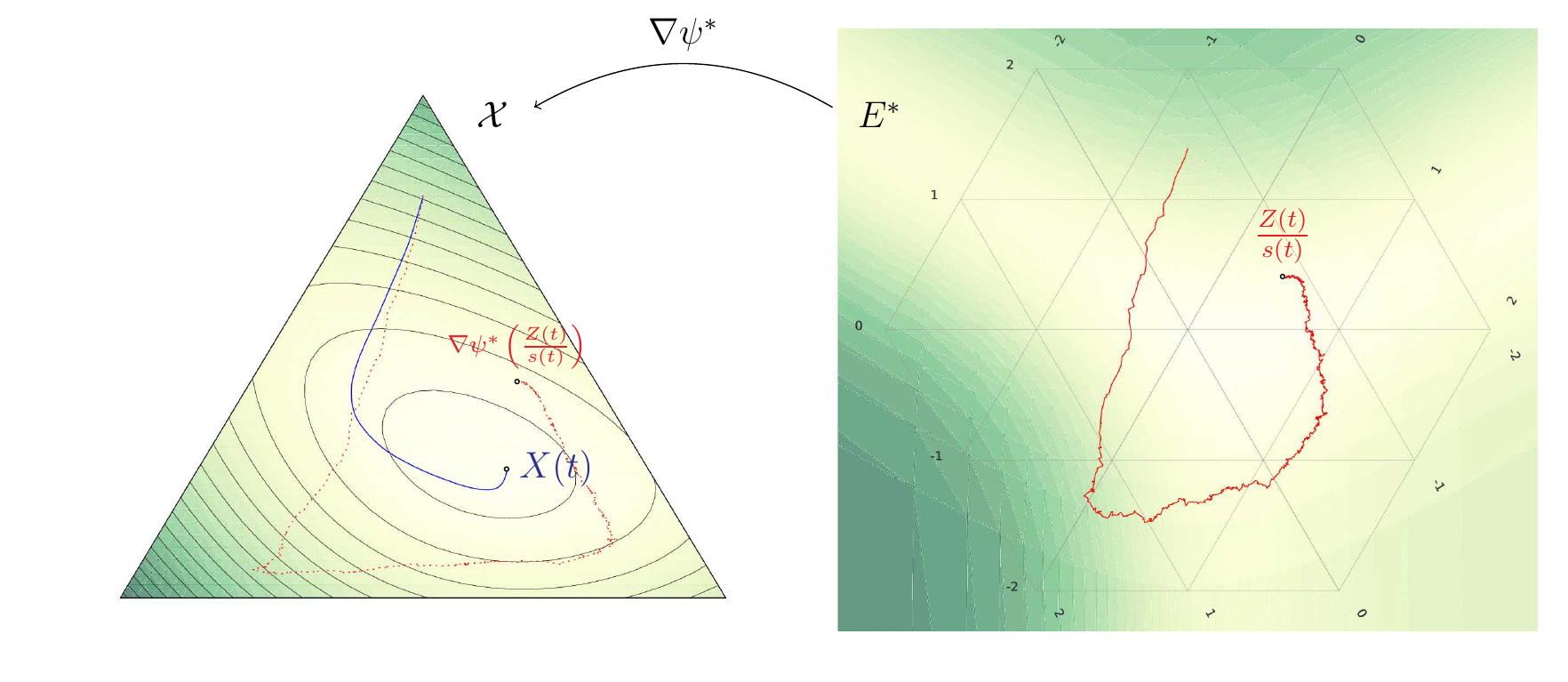}
\caption{Illustration of SAMD dynamics. The dual trajectory $Z(t)$ cumulates negative gradients, $\dot Z(t) = -\eta(t)\nabla f(X(t))$. We visualize the scaled dual trajectory $Z(t) / s(t)$ in the dual space (red), and the corresponding mirror $\nabla \psi^*(\frac{Z(t)}{s(t)})$ in the primal space (dotted red). The primal trajectory $X(t)$ is obtained by averaging the mirror.}
\label{fig:samd}
\end{figure}

We give an illustration of SAMD dynamics in Figure~\ref{fig:samd}, on a simplex-constrained problem in $\Rbb^3$. The feasible set is given by $\Xcal = \{x \in \Rbb^n_+ : \sum_{i = 1}^n x_i = 1\}$, and the mirror map is generated by the negative entropy restricted to the simplex, given by:
\[
\psi(x) = \begin{cases}
-\sum_{i = 1}^n x_i \ln x_i & \text{if $x \in \Xcal$,} \\
+\infty & \text{otherwise,}
\end{cases}
\]
which is strongly convex with respect to the norm $\|\cdot\|_1$ by Pinsker's inequality. Its convex conjugate is given by
\al{
\psi^*(z) 
= \max_{x \in \Xcal} \braket{x}{z} - \sum_{i = 1}^n x_i \ln x_i = \ln \sum_{i = 1}^n e^{z_i}
}
which is differentiable for all $z \in E^*$, and the mirror map is
\[
\nabla \psi^*(z) = \frac{e^{z_i}}{\sum_{j = 1}^n e^{z_j}} \in \Xcal
\]
which is Lipschitz w.r.t. the dual norms $\|\cdot\|_1, \|\cdot\|_\infty$.

Note that for all $z \in E^*$ and all $\alpha \in \Rbb$,
\[
\nabla \psi^*(z) = \nabla \psi^*(z + \alpha \mathbf 1),
\]
where $\mathbf 1$ is the vector of all ones. This can be verified directly using the expression of $\nabla \psi^*$, but can also be seen as a consequence of the duality of sub-differentials (e.g. Theorem 23.5 in \cite{rockafellar1970convex}), which states that $x = \nabla \psi^*(z)$ if and only if $z \in \partial \psi(x)$; and since $\psi$ is the restriction of the negative entropy $-H(x) = -\sum_{i = 1}^nx_i \ln x_i$ to the simplex, its sub-differential at $x$ is
\[
\partial \psi(x) = -\nabla H(x) + n_\Xcal(x)
\]
where $n_\Xcal(x)$ is the normal cone to $\Xcal$ at $x$, which is simply the line $\Rbb \mathbf 1$ (when $x$ is in the relative interior of the simplex).

Since the mirror map is constant along the normal to the simplex, we choose to project the dual variable $Z$ on the hyperplane parallel to the simplex, for visualization purposes. This allows us to visualize the relevant component of the dual dynamics, and ignore a component which does not matter for convergence (but which could have high magnitudes if $\nabla f$ has a large component along the normal). Note that even numerically, projecting $Z$ after each iteration helps improve numerical stability (without affecting the primal trajectory).

Finally in order to visualize the function values, we generate a triangular mesh of the simplex, then map it to the dual space. In other words, the colors in the primal space represent $f(x)$, and in the dual space represent $f(\nabla \psi^*(z))$. It is interesting to observe how the mirror map $\nabla \psi^*$ distorts the space between primal and dual spaces.

The objective function used in this example (which we use as a running numerical example to illustrate our results) is given by the sum of exponentials
\[
f(x) = \sum_{i = 1}^k e^{\braket{c_i}{x}},
\]
where $\{c_i\}_{1 \leq i \leq k}$ are vectors in $\Rbb^n$.

\subsection{Existence and uniqueness of a continuous solution}
We give the following existence and uniqueness result.
\begin{proposition}
Suppose Assumptions~\ref{assumption:lip}, \ref{assumption:compact} and \ref{assumption:volatility} hold. Then for all $T > t_0$, $\SAMD_{\eta, a, s}$ has a unique (up to redefinition on a $\mathbb P$-null set) solution $(X(t), Z(t))$ continuous on $[0, T]$, with the property that $(X(t), Z(t))$ is adapted to the filtration $\{\Fcal_t\}$, and $\int_{t_0}^T \|X(t)\|^2 d_t, \int_{t_0}^T \|Z(t)\|_*^2 dt$ have finite expectations.
\end{proposition}
\begin{proof}
By assumption, $\nabla \psi^*$ and $\nabla f$ are Lipschitz continuous, thus the function $(x, z) \mapsto (-\eta(t)\nabla f(x), a(t)[\nabla \psi^*(z/s(t)) - x])$ is Lipschitz on $[t_0, T]$ (since $a, \eta, s$ are positive continuous). Additionally, the function $x \mapsto \sigma(x, t)$ is Lipschitz on the compact set $\Xcal$. The function $\sigma_*(t)$ is continuous, since it is by definition the supremum of continuous functions
\[
\sigma_*^2(t) = \sup_{x \in \Xcal} \|\Sigma(x, t)\|_i = \sup_{x \in \Xcal} \sup_{\|z\|_* \leq 1} \|\Sigma(x, t)z\|
\]
and $\Xcal$ and $\{z \in E^* : \|z\|_* \leq 1\}$ are both compact.

Therefore, we can invoke the existence and uniqueness theorem for stochastic differential equations~\citep[Theorem 5.2.1]{oksendal2003stochastic}.
\end{proof}
Note that since $T$ is arbitrary in the previous proposition, we can conclude that there exists a unique continuous solution on $[t_0, \infty)$. In the previous proposition, we assumed that $\Xcal$ is compact to guarantee the conditions of the existence and uniqueness theorem, but this can be relaxed. Instead, we can directly assume additional regularity of the noise, e.g. that $\Sigma(x, t)$ additionally satisfies $\|\Sigma(x, t)\|_i \leq C(1+\|x\|)$ for some $C > 0$, uniformly in $t$.

Next, in order to analyze the convergence properties of the solution trajectories $(X(t), Z(t))$, we will need to bound the time-derivative of the energy function $L$.

\subsection{Energy decay}
In this section, we state and prove two technical lemmas which will be useful in proving our main convergence results. We start by bounding the rate of change of the energy function along solution trajectories of $\SAMD$.
\subsubsection*{Bounding the rate of change of the energy}
\begin{lemma}
\label{lem:lyap_bound_sto}
Suppose Assumption~\ref{assumption:lip} holds, and that the primal rate satisfies $a = \eta / r$, and let $(X(t), Z(t))$ be a continuous solution to $\SAMD_{\eta, \eta/r, s}$. Then for all $t \geq t_0$,
\al{
dL&(X(t), Z(t), t) \leq \\
&\left[(f(X(t)) - f(x^\star))(\dot r(t) - \eta(t)) + \psi(x^\star) \dot s(t) + \frac{n L_{\psi^*}}{2} \frac{\eta^2(t)\sigma_*^2(t)}{s(t)}\right] dt + \braket{V(t)}{dB(t)},
}
where $V(t)$ is the continuous, $n$ dimensional process given by
\begin{equation}
\label{eq:v}
V(t) = -\eta(t) \sigma(X(t), t)^T(\nabla \psi^*(Z(t)/s(t)) - \nabla \psi^*(z^\star))
\end{equation}
\end{lemma}

\begin{proof}
By definition of the energy function $L$, $\nabla_x L(x, z, t) = r(t) \nabla f(x)$ and $\nabla_z L(x, z, t) = \nabla \psi^*(z / s(t)) - \nabla \psi^*(z^\star)$, which are Lipschitz continuous in $(x, z)$ (uniformly in $t$ on any bounded interval, since $s(t), r(t)$ are continuous positive functions of $t$). Thus by the It\^o formula for functions with Lipschitz continuous gradients~\citep{errami2002ito}, we have
\al{
dL 
&= \partial_t L dt + \braket{\nabla_x L}{dX} + \braket{\nabla_z L}{dZ} + \frac{1}{2}\tr\parenth{\eta\sigma^T\nabla^2_{zz} L \sigma \eta} dt \\
&= \partial_t L dt + \braket{\nabla_x L}{dX} + \braket{\nabla_z L}{-\eta \nabla f(X)} dt + \braket{\nabla_z L}{-\eta \sigma dB} + \frac{\eta^2}{2}\tr\parenth{\Sigma \nabla^2_{zz} L } dt.
}
The first three terms correspond exactly to the deterministic case, and we can bound them by~\eqref{eq:lyap_bound} from Lemma~\ref{lem:lyap_bound}. The last two terms are due to the stochastic noise, and consist of a volatility term
\[
-\eta \braket{\nabla_z L(X, Z, t)}{\sigma dB} = -\eta \braket{\nabla \psi^*(Z/s) - \nabla \psi^*(z^\star)}{\sigma dB} = \braket{V}{dB},
\]
and the It\^o correction term
\[
\frac{\eta^2}{2} \tr\parenth{\Sigma(X, t) \nabla^2_{zz} L(X, Z, t) } dt = \frac{\eta^2}{2s} \tr\parenth{\Sigma(X, t) \nabla^2 \psi^*(Z/s) } dt.
\]
We can bound the last term using the following simple fact: Given two linear operators $P: E \to E^*$ and $Q: E^* \to E$, such that $\|P\|_{*,i} \leq \alpha_P$ and $\|Q\|_{i} \leq \alpha_Q$ (where $\|\cdot\|_{i}$ and $\|\cdot\|_{*,i}$ are the norms induced by the pair of dual norms $\|\cdot\|$ and $\|\cdot\|_*$), then we have $\tr(PQ) \leq n \alpha_P \alpha_Q$ since
\al{
\tr(PQ)
= \sum_{j = 1}^n \braket{P_j}{Q_j}
\leq \sum_{j = 1}^n \|P_j\|_* \|Q_j\|
\leq \sum_{j = 1}^n \|P\|_{*, i} \|Q\|_{*, i}.
}
Now, since $\nabla \psi^*$ is, by assumption, $L_{\psi^*}$-Lipschitz, we have $\|\nabla \psi^*\|_i \leq L_{\psi^*}$, and by definition~\eqref{eq:sigma_star} of~$\sigma^*$, $\|\Sigma(x, t)\|_* \leq \sigma_*^2(t)$ for all $x$, therefore for all $x \in E$, $z \in E^*$, and $t \geq t_0$, 
\[
\tr(\Sigma(x, t) \nabla^2 \psi^*(z)) \leq n L_{\psi^*} \sigma_*^2(t).
\]
Combining the previous inequalities, we obtain the desired bound.
\end{proof}

Comparing the bound of Lemma~\ref{lem:lyap_bound_sto} to its deterministic counterpart of Lemma~\ref{lem:lyap_bound}, we can see that in the stochastic case, we have two additional terms: a drift term proportional to $\sigma_*^2(t)$, which is due to the It\^o correction term, and a volatility term given by $\braket{V(t)}{dB(t)}$.

To illustrate these terms, we generate solution trajectories for the sum-exponential example of Section~\ref{sec:example}, for both the deterministic and stochastic dynamics, and plot the values of the objective function and the energy function. For simplicity, we consider a constant volatility $\sigma(x, t) = \sigma_0 = .1$, a linear energy rate $r(t) = t$, and a sensitivity $\frac{1}{s(t)} = \frac{1}{\sqrt t}$. The primal and dual weights are $\eta(t) = \dot r(t) = 1$ and $a(t) = \frac{\eta(t)}{r(t)} = \frac{1}{t}$. For the stochastic dynamics, we plot the mean and standard deviation of $100$ sample trajectories. The results are given in Figure~\ref{fig:AMD_SAMD} below.

\begin{figure}[h]
\centering
\includegraphics[width=\textwidth]{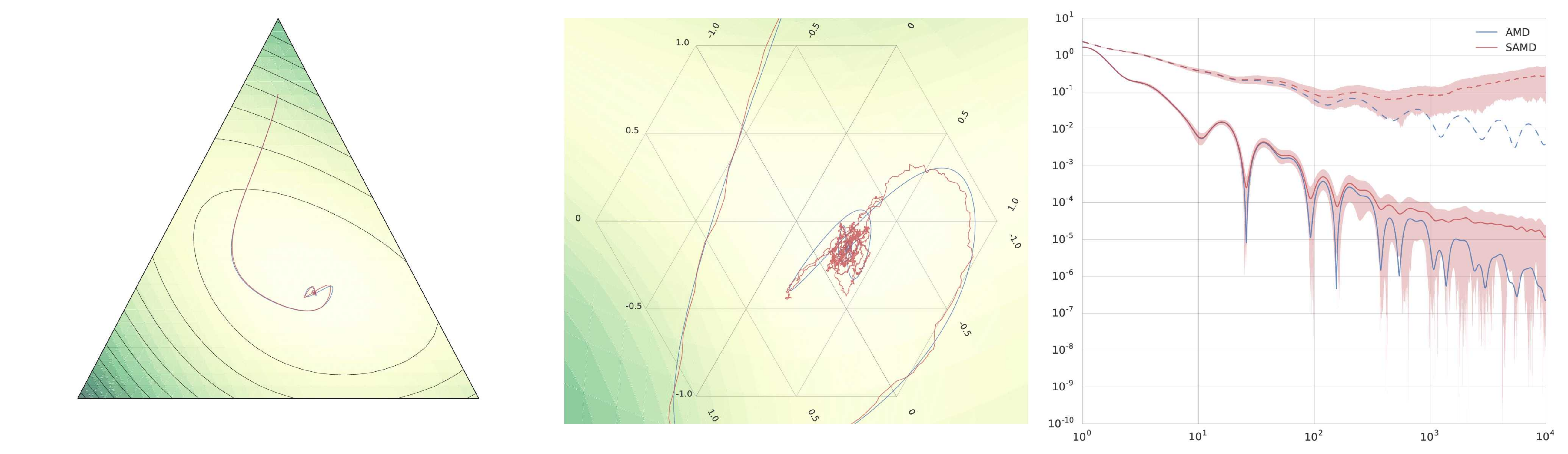}
\caption{Illustration of the differences between the deterministic and stochastic variants of accelerated mirror descent dynamics. The right plot shows the values of the objective function $f(X(t)) - f(x^\star)$ (solid lines) and the energy function $L(X(t), Z(t), t)$ (dashed lines) as a function of $t$. For the stochastic dynamics, we generate $100$ sample trajectories and plot the mean and standard deviation. The difference in mean energy values illustrates the drift term (captured by the It\^o correction term), and the standard deviation illustrates the volatility term. We also plot the primal and dual trajectories (for a single sample) in the left and center figures, respectively.}
\label{fig:AMD_SAMD}
\end{figure}
\begin{figure}[h]
\centering
\includegraphics[width=\textwidth]{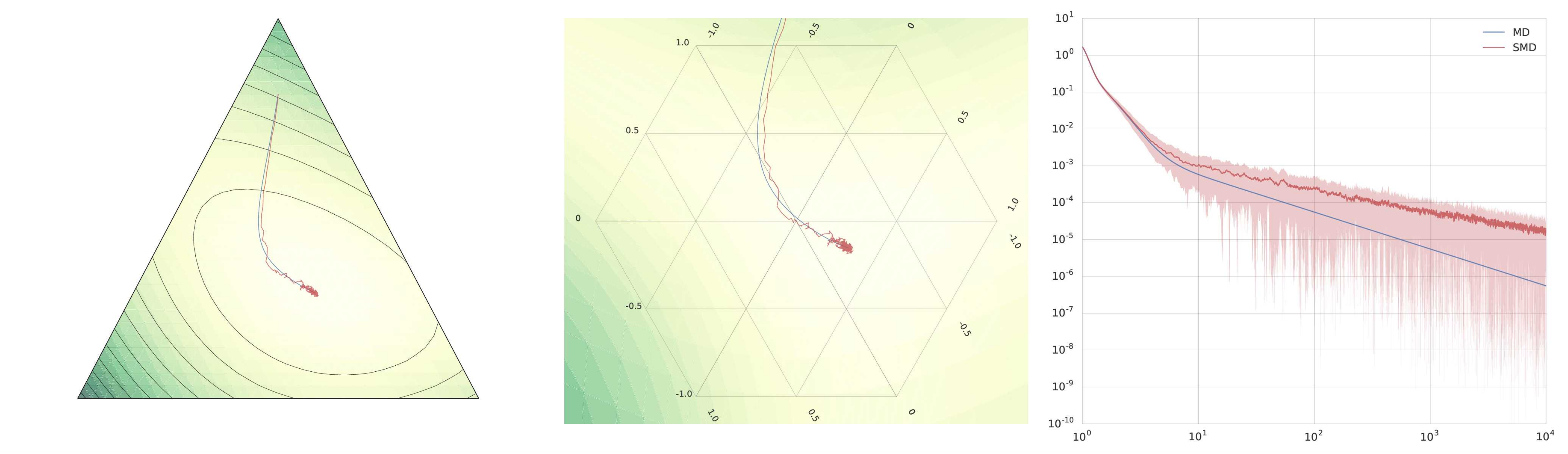}
\caption{Illustration of the differences between the deterministic and stochastic variants of (non-accelerated) mirror descent dynamics. The stochastic dynamics exhibit a slower convergence rate.}
\label{fig:MD_SMD}
\end{figure}

For the sake of comparison, we also generate a similar simulation for (non-accelerated) mirror descent dynamics in Figure~\ref{fig:MD_SMD}, using the same sensitivity $\frac{1}{s(t)} = \frac{1}{\sqrt{t}}$. Comparing Figures~\ref{fig:AMD_SAMD} and~\ref{fig:MD_SMD}, we can already observe some qualitative differences introduced by acceleration: the accelerated dynamics exhibit faster convergence, accompanied with typical oscillations around the minimizer. Note that these oscillations are not due to discretization or noise in the gradient estimate (as can be observed e.g. in discrete gradient descent with large step sizes). The oscillations are rather a property of the continuous-time accelerated dynamics. Besides, acceleration also seems to reduce the effect of noise on the primal trajectory (which appears visually smoother than its non-accelerated counterpart). To give some intuition, consider the following informal argument (which will be formalized in the results of Section~\ref{sec:convergence}). In the case of mirror descent, the primal variable is obtained as the mirror of $Z(t)/s(t)$, where
\al{
\frac{Z(t)}{s(t)} = -\frac{1}{s(t)}\int_{t_0}^t \nabla f(X(\tau))d\tau + \frac{1}{s(t)}\int_{t_0}^t \sigma(X(\tau), \tau) dB(\tau),
}
and the noise is cumulated in the It\^o martingale term $\int_{t_0}^t \sigma(X(\tau), \tau) dB(\tau)$. In the case of accelerated mirror descent, there are two main differences:
\begin{enumerate}
\item First, using a time-varying dual rate $\eta(t)$ in $\SAMD$ leads to a non-linear accumulation of noise in the It\^o martingale $\int_{t_0}^t \eta(\tau)\sigma(X(\tau), \tau)dB(\tau)$.
\item Second, due to the averaging in the primal space (given by the integral~\eqref{eq:primal_integral_form}), the smaller noise in the past trajectory results in a smaller noise of the weighted average (similarly to acceleration by averaging of~\cite{polyak1992acceleration} in discrete time). 
\end{enumerate}
The combined effect of averaging in the dual space (using $\eta(t)$) and in the primal space (using $a(t)$) will be apparent in Section~\ref{sec:convergence} when we derive explicit convergence rates.

\subsubsection*{Bounding the It\^o martingale term}
Integrating the bound of Lemma~\ref{lem:lyap_bound_sto} will allow us to bound changes in energy. This bound will involve the It\^o martingale $\int_{t_0}^t \braket{V(\tau)}{dB(\tau)}$, where $V$ is defined in~\eqref{eq:v}. In order to control this term, we give, in the following lemma, an asymptotic envelope (a consequence of the law of the iterated logarithm).
\begin{lemma}
\label{lem:martingale_bound}
Suppose Assumptions~\ref{assumption:compact} and \ref{assumption:volatility} hold. Let $b(t) = \int_{t_0}^t \eta^2(\tau)\sigma_*^2(\tau) d\tau$. Then
\begin{align}
\int_{t_0}^t \braket{V(\tau)}{dB(\tau)} = \Ocal(\sqrt{b(t) \log\log b(t)}) && \text{a.s. as $t \to \infty$}.
\end{align}
\end{lemma}

\begin{proof}
Let us denote the It\^o martingale by $\Vcal(t) = \int_{t_0}^t \braket{V(\tau)}{dB(\tau)} = \sum_{i = 1}^n \int_{t_0}^t V_i(\tau)dB_i(\tau)$, and its quadratic variation by $\beta(t) = [\Vcal(t), \Vcal(t)]$. By definition of $\Vcal$, we have
\[
d\beta = \sum_{i = 1}^n\sum_{j = 1}^n V_iV_jd[B_i, B_j] = \sum_{i = 1}^n V_i^2 dt = \braket{V}{V} dt.
\]
By the Dambis-Dubins-Schwartz time change theorem (e.g. Corollary 8.5.4 in~\citep{oksendal2003stochastic}), there exists a Wiener process $\hat B$ such that
\begin{equation}
\label{eq:time_change}
\Vcal(t) = \hat B(\beta(t)).
\end{equation}
We now proceed to bound $\beta(t)$. Using the expression~\eqref{eq:v} of $V$, we have
\[
\braket{V}{V} = \eta^2(t) \Delta^T(t) \Sigma(X, t) \Delta(t),
\]
where we defined $\Delta(t) = \nabla \psi^*(Z(t)/s(t)) - \nabla \psi^*(z^\star)$. Since the mirror map has values in $\Xcal$ and $\Xcal$ is assumed compact, the diameter $D = \sup_{x, x' \in \Xcal} \|x - x'\|$ is finite, and $\Delta(t) \leq D$ for all $t$. Thus, $d\beta(t) \leq D^2 \eta(t)^2 \sigma_*^2(t) dt$, and integrating,
\begin{align}
\label{eq:covariation_bound}
\beta(t) \leq D^2 b(t) & \text{\ \ a.s.}
\end{align}
Since $\beta(t)$ is a non-decreasing process, two cases are possible: if $\lim_{t \to \infty} \beta(t)$ is finite, then $\limsup_{t \to \infty}|\Vcal(t)|$ is a.s. finite and the result follows immediately. If $\lim_{t \to \infty} \beta(t) = \infty$, then
\al{
\limsup_{t \to \infty} \frac{\Vcal(t)}{\sqrt{b(t) \log \log b(t)}} \leq \limsup_{t \to \infty} \frac{\hat B(\beta(t))}{\sqrt{\frac{\beta(t)}{D^2} \log \log \frac{\beta(t)}{D^2}}} = D \sqrt 2 && \text{a.s.}
}
where the inequality combines~\eqref{eq:time_change} and \eqref{eq:covariation_bound}, and the equality is by the law of the iterated logarithm.
\end{proof}

\section{Convergence results}
\label{sec:convergence}

\subsection{Almost sure convergence}
Equipped with Lemma~\ref{lem:lyap_bound_sto} and Lemma~\ref{lem:martingale_bound}, which bound, respectively, the rate of change of the energy and the asymptotic growth of the martingale term, we are now ready to prove our convergence results.
\begin{theorem}
\label{thm:as}
Suppose that Assumptions~\ref{assumption:lip}, \ref{assumption:compact}, and \ref{assumption:volatility} hold. Suppose that $\eta(t) \sigma_*(t) = o(1/\sqrt{\log t})$, and that $\int_{t_0}^t \eta(\tau)d\tau$ dominates $b(t)$ and $\sqrt{b(t) \log\log b(t)}$ (where $b(t) = \int_{t_0}^t \eta^2(\tau) \sigma_*^2(\tau)d\tau$ as defined in Lemma~\ref{lem:martingale_bound}). Consider $\SAMD$ dynamics with $r = s = 1$ and $a = \eta$. Let $(X(t), Z(t))$ be the unique continuous solution of $\SAMD_{\eta, \eta, 1}$. Then
\al{
\lim_{t \to \infty} f(X(t)) - f(x^\star) = 0 &&\text{a.s.}
}
\end{theorem}

The result of Theorem~\ref{thm:as} makes it possible to guarantee almost sure convergence (albeit without an explicit convergence rate) when the noise is persistent ($\sigma_*(t)$ is constant, or even increasing). To give a concrete example, suppose $\sigma_*(t) = \Ocal(t^{\alpha})$ (with $\alpha < \frac{1}{2}$ but can be positive), and let $\eta(t) = t^{-\alpha - \frac{1}{2}}$. Then $\eta(t)\sigma_*(t) = \Ocal(t^{-\frac{1}{2}})$, $\int_{t_0}^t \eta(\tau)d\tau = \Omega(t^{-\alpha+\frac{1}{2}})$, $b(t) = \Ocal(\log t)$, and $\sqrt{b(t) \log \log b(t)} = \Ocal(\sqrt{\log t \log \log \log t})$, and the conditions of the theorem are satisfied. Therefore, with the appropriate choice of learning rate $\eta(t)$ (and the corresponding averaging in the primal space given by $a(t) = \eta(t)$), one can recover almost sure convergence.

We start by giving an outline of the proof, which is similar to that of Theorem~4.1 in \citep{mertikopoulos2016convergence}, with some significant changes (we do not make the assumption that the minimizer is unique, and most importantly, the dynamics and the energy function are different, since averaging is essential in our case to handle the noise, since we do not assume that the volatility bound $\sigma_*(t)$ is vanishing). The argument proceeds in the following steps:

\begin{enumerate}[(i)]
\item The first step is to prove that under the conditions of the theorem, the continuous solution of $\SAMD_{\eta, \eta, 1}$, $(X(t), Z(t))$, is an asymptotic pseudo trajectory (a notion defined and studied by~\cite{benaim1996asymptotic} and \cite{benaim1999dynamics}) of the deterministic flow $\AMD_{\eta, \eta, 1}$. The definition is given below, but intuitively, this means that for large enough times, the sample paths of the process $(X(t), Z(t))$ get arbitrarily close to $(x(t), z(t))$, the solution trajectories of the deterministic dynamics.

\begin{definition}[Asymptotic Pseudo Trajectory]
Let $\Phi_t : \Xcal \times E^* \to \Xcal \times E^*$ be the semi-flow associated to the deterministic dynamics $\AMD_{\eta, \eta, 1}$, that is, $(x(t), z(t)) = \Phi_t(x_0, z_0)$ is the solution of the deterministic dynamics $\AMD_{\eta, \eta, 1}$ with initial condition $(x_0, z_0)$. A continuous function $t \mapsto (X(t), Z(t)) \in \Xcal \times E^*$ is an asymptotic pseudo trajectory (APT) for $\Phi_t$ if for all $T > 0$,
\[
\lim_{t \to \infty} \sup_{0 \leq h \leq T} d((X(t+h), Z(t+h)), \Phi_h(X(t), Z(t))) = 0,
\]
where $d$ is a distance on $\Xcal\times E^*$, e.g. $d((x, z), (x', z')) = \|x - x'\| + \|z - z'\|_*$.
\end{definition}

\item The second step is to show that under the deterministic flow, the energy $L$ decreases enough for large enough times.

\item The third step is to prove that under the stochastic process, $f(X(t))$ cannot stay bounded away from $f(x^\star)$ for all~$t$.
\end{enumerate}
Finally, combining these steps, we argue that by (iii), $f(X(t))$ eventually gets close to $f(x^\star)$, then stays close by virtue of the asymptotic pseudo trajectory property (i), and the decrease of the energy under the deterministic flow (ii).


\begin{proof}[Proof of Theorem~\ref{thm:as}]
We start by specializing the energy function and the bounds on its time derivative to the setting of Theorem~\ref{thm:as}. Under the assumptions of the theorem ($r(t) = s(t) = 1$), $L(x, z, t)$ simplifies to
\[
L_{z^\star}(x, z) = f(x) - f(x^\star) + D_{\psi^*}(z, z^\star),
\]
where we added the subscript $z_\star$ to insist on the fact that the energy function is ``anchored'' at $z^\star$. Note that since the minimizer is not necessarily unique, $L_{z^\star}(x(t), z(t))$ does not necessarily converge to $0$ for arbitrary $z^\star$. Thus, we define and use
\[
\bar L(x, z) = \inf_{z^\star \in \Zcal^\star} L_{z^\star}(x, z),
\]
where $\Zcal^\star = \{z \in E^* : \nabla \psi^*(z^\star) \in \Xcal^\star\} = \cup_{x^\star \in \Xcal^\star} \partial \psi(x^\star)$ (by the fact that $x^\star \in \partial \psi^*(z^\star)$ if and only if $z^\star \in \partial \psi(x^\star)$).

Next, we observe that since $\nabla f$ is $L_f$-Lipschitz and $\nabla \psi^*$ is $L_{\psi^*}$-Lipschitz, we can bound the change of the energy due to small displacements in $(x, z)$: we will use the fact that for any convex function $f$ with $L$-Lipschitz gradient, $f(x+\delta_x) \leq f(x) + \braket{\nabla f(x)}{\delta_x} + \frac{L}{2}\|\delta_x\|^2$. We have
\begin{align}
L_{z^\star}&(x+\delta_x, z + \delta_z) \notag \\
&= f(x+\delta_x) - f(x^\star) + \psi^*(z+\delta_z) - \psi^*(z^\star) - \braket{\nabla \psi^*(z^\star)}{z+\delta z - z^\star} \notag\\
&\leq f(x) + \braket{\nabla f(x)}{\delta_x} + \frac{L_f}{2}\|\delta_x\|^2 - f(x^\star) \notag\\
&\quad + \psi^*(z) + \braket{\nabla \psi^*(z)}{\delta z} + \frac{L_{\psi^*}}{2}\|\delta_z\|_*^2 - \psi^*(z^\star) -\braket{\nabla \psi^*(z^\star)}{z+\delta z - z^\star} \notag\\
&= L_{z^\star}(x, z) + \braket{\nabla f(x)}{\delta_x} + \frac{L_f}{2}\|\delta_x\|^2 + \braket{\nabla \psi^*(z) - \nabla \psi^*(z^\star)}{\delta z} + \frac{L_{\psi^*}}{2}\|\delta_z\|_*^2 \notag\\
&\leq L_{z^\star}(x, z) + G\|\delta_x\| + \frac{L_f}{2}\|\delta_x\|^2 + D\|\delta z\|_* + \frac{L_{\psi^*}}{2}\|\delta_z\|_*^2 \label{eq:lyap_lip}
\end{align}
where in the last inequality, $G = \sup_{x \in \Xcal} \|\nabla f(x)\|_*$ (which is bounded since $\nabla f$ is continuous and $\Xcal$ is compact), and $D$ is the diameter of $\Xcal$.

For the deterministic dynamics, the bound of Lemma~\ref{lem:lyap_bound} becomes
\begin{equation}
\label{eq:lyap_bound}
\frac{d}{dt} L_{z^\star}(x(t), z(t)) \leq -\eta(t)(f(x(t)) - f(x^\star)),
\end{equation}
and for the stochastic dynamics, the bound of Lemma~\ref{lem:lyap_bound_sto} becomes
\begin{equation}
\label{eq:lyap_bound_sto}
dL_{z^\star}(X(t), Z(t)) \leq \left[-\eta(t)(f(X(t)) - f(x^\star)) + \frac{L_{\psi^*}}{2} \eta^2(t)\sigma_*^2(t) \right] dt + \braket{V(t)}{dB(t)}.
\end{equation}

We now proceed according to the steps of the proof outline. We give an illustration of the argument in Figure~\ref{fig:APT_illustration}
\begin{enumerate}[(i)]
\item We start by proving that under the conditions of Theorem~\ref{thm:as}, the stochastic process $(X(t), Z(t))$ (the unique continuous solution of the stochastic dynamics $\SAMD_{\eta, \eta, 1}$) is an APT for the deterministic semi-flow of $\AMD_{\eta, \eta, 1}$. Since the volatility term is $-\eta(t) \sigma(X(t), t) dB(t)$, it suffices, by Proposition~4.6\footnote{Proposition 4.6 in \citep{benaim1999dynamics} is stated in terms of solutions to a martingale problem, which is equivalent to solutions to the SDE, see for example~\citep{stroock1972}.} in~\cite{benaim1999dynamics}, to show that $\int_{t_0}^\infty e^{-\frac{c}{\eta^2(t)\sigma_*^2(t)}}$ is finite for all $c>0$. But we have, by assumption, $\eta(t)\sigma_*(t) = o(1/\sqrt{\log t})$, thus $\eta^2(t)\sigma_*^2(t) = \epsilon(t) / \log t$ with $\lim_{t \to \infty} \epsilon(t) = 0$, and $\int_{t_0}^{\infty} e^{-\frac{c}{\eta^2(t)\sigma_*^2(t)}}dt = \int_{t_0}^{\infty} e^{-\frac{c \log t}{\epsilon(t)}}dt = \int_{t_0}^{\infty} t^{-\frac{c}{\epsilon(t)}}dt$, which is finite.

We also show that by virtue of the APT property (and the fact that the energy function is Lipschitz), we can bound the difference between the energy $\bar L$ along deterministic and stochastic solutions starting at the same point. Indeed, inequality~\eqref{eq:lyap_lip} shows that $L_{z^\star}(x+\delta_x, z + \delta_z) - L_{z^\star}(x, z) \leq \epsilon$ whenever $\max(\|\delta_x\|, \|\delta_z\|_*)$ is small enough. Therefore, by the APT property, for all $\epsilon > 0$ and all $T>0$, there exists $t_T$ such that for all $t \geq t_T$ and all $h \in [0, T]$,
\[
L_{z^\star}(X(t+h), Z(t+h)) - L_{z^\star}(\Phi_h(X(t), Z(t))) \leq \epsilon/2,
\]
and this holds uniformly over $z^\star$. In particular, since $\bar L$ is defined to be the infimum over all $z^\star$, we can find some $z^\star_0$ such that $L_{z^\star_0}(\Phi_h(X(t), Z(t))) \leq \bar L(\Phi_h(X(t), Z(t))) + \epsilon/2$, then
\al{
\bar L(X(t+h), Z(t+h)) 
&\leq L_{z^\star_0}(X(t+h), Z(t+h)) \\
&\leq L_{z^\star_0}(\Phi_h(X(t), Z(t))) + \epsilon/2 \\
&\leq \bar L(\Phi_h(X(t), Z(t))) + \epsilon.
}

\item Next, we prove a stability property of the energy for the deterministic dynamics. Fix $\epsilon > 0$ and let $V_\epsilon = \{(x, z) : \bar L(x, z) \leq \epsilon\}$. Then $\Phi_t(x, z) \in V_\epsilon$ if $(x, z) \in V_\epsilon$ (since $\bar L$ is non-increasing, as the infimum of non-increasing functions). Besides, we claim that there exists $T > 0$ such that for all $t \geq T$,
\[
\bar L(\Phi_t(x, z)) \leq \min(\epsilon, \bar L(x, z) - \epsilon).
\]
Indeed, by continuity of $f$, there exists $c > 0$ such that $f(x) - f(x^\star) > c$ for all $(x, z) \notin V_{\epsilon}$, and integrating the bound~\eqref{eq:lyap_bound}, we have, for all $z^\star$,
\[
L_{z^\star}(\Phi_t(x, z)) \leq L_{z^\star}(x, z) -c\int_{T}^t \eta(\tau)d\tau,
\]
therefore, setting $T = 2\epsilon / c$, we know that either $\Phi_t(x, z) \in V_\epsilon$ for some $t_1 \leq T$, in which case the trajectory remains in $V_\epsilon$ after $t_1$, or $\Phi_t(x, z)$ remains outside of $V_\epsilon$, in which case $L_{z^\star}(\Phi_T(x, z)) \leq L_{z^\star}(x, z) -2\epsilon$ for all $z^\star$. Since $\bar L$ is defined to be the infimum over all $z^\star$, we can find some $z^\star_0$ such that $L_{z^\star_0}(x, z) \leq \bar L(x, z) + \epsilon$. Then
\[
\bar L(\Phi_T(x, z)) \leq L_{z^\star_0}(\Phi_T(x, z)) \leq L_{z^\star_0}(x, z) - 2\epsilon \leq \bar L(x, z) - \epsilon.
\]
\item Next, we prove that the stochastic process cannot stay outside of $V_\epsilon$ for unbounded intervals of time. Indeed, fix $\epsilon > 0$, and $T > 0$, and suppose that with positive probability, $(X(t), Z(t))$ remains outside $V_\epsilon$ for all $t \geq T$. Then by continuity of $f$, there exists $c>0$ such that $f(X(t)) - f(x^\star) \geq c$ for all $t \geq T$, and integrating the bound~\eqref{eq:lyap_bound_sto} gives
\[
L_{z^\star}(X(t), Z(t)) - L_{z^\star}(X(T), Z(T)) \leq -c \int_{T}^t \eta(\tau) d\tau + \Ocal(b(t)) + \Ocal(\sqrt{b(t) \log\log b(t)}),
\]
where the right-hand side converges to $-\infty$ since, by assumption, $\int_{t_0}^t \eta(\tau)d\tau$ dominates $b(t)$ and $\sqrt{b(t)\log\log b(t)}$. This would imply that, with positive probability, $L(X(t), Z(t), t) \to -\infty$, a contradiction. Therefore, for all $\epsilon > 0$ and for all $T$, there exists $t \geq T$ such that $(X(T), Z(T)) \in V_\epsilon$ a.s.
\end{enumerate}

\begin{figure}[h]
\centering
\includegraphics[width=\textwidth]{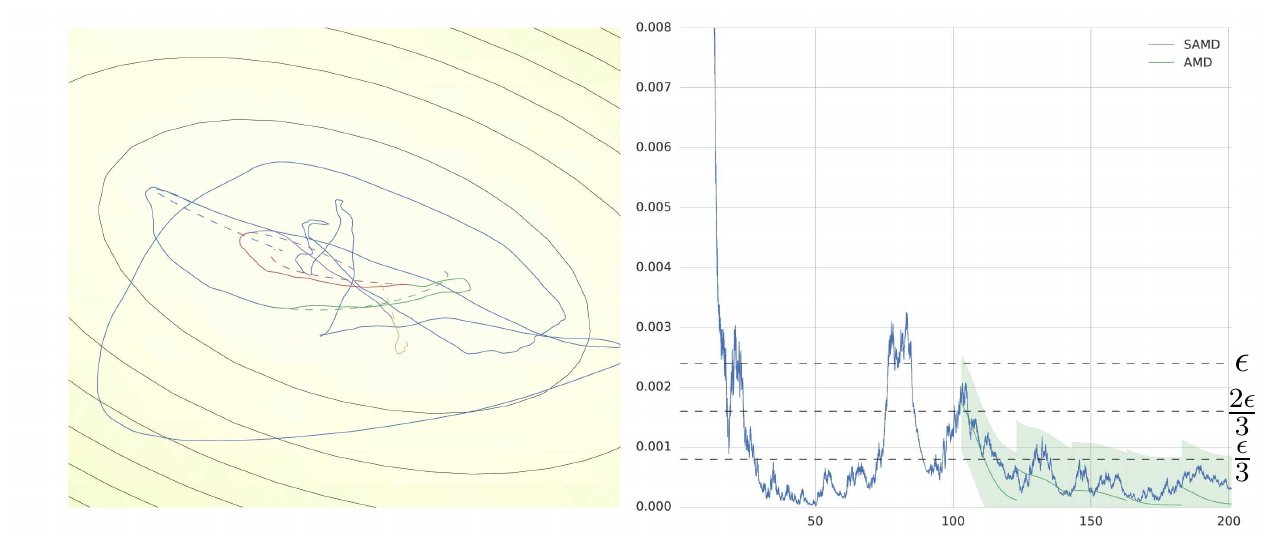}
\caption{Illustration of the proof of Theorem~\ref{thm:as}, with $\epsilon = 2.4 \ 10^{-3}$, and $T_0 = 20$. The right plot shows in blue the value of the energy function $\bar L(X(t), Z(t))$ along one sample trajectory $(X(t), Z(t))$ of the SAMD dynamics; and in green the energy function along solutions of the deterministic ODE $\{(x_k(t), z_k(t)), {t \in [T_2 + kT_0, T_2 + (k+1)T_0]}\}$, initialized at $(X(T_2 + kT_0), Z(T_2 + kT_0))$. We also highlight a cylinder of radius $\frac{\epsilon}{3}$ centered at the deterministic energy. Note that for large enough times, the sample path of the stochastic dynamics remains within the cylinder. The dashed lines show the energy levels $\frac{\epsilon}{3}$, $\frac{2\epsilon}{3}$, and $\epsilon$. Finally, the left plot visualizes these trajectories in the primal space (where we used a different color for each interval $[T_2 + kT_0, T_2 + (k+1)T_0]$).}
\label{fig:APT_illustration}
\end{figure}

We are now ready to put together the different parts of the argument. Fix $\epsilon > 0$.\\
By (ii), there exists $T_0$ such that
\begin{equation}
\label{eq:proof_1}
\bar L(\Phi_{T_0}(x), \Phi_{T_0}(z)) \leq \min(\epsilon/3, \bar L(x, z) - \epsilon/3).
\end{equation}
By (i) there exists $T_1$ such that for $t \geq T_1$ and for all $h \in [0, T_0]$,
\begin{equation}
\label{eq:proof_2}
\bar L(X(t+h), Z(t+h)) \leq \bar L(\Phi_h(X(t), Z(t))) + \epsilon/3,
\end{equation}
By (iii), there exists $T_2 \geq \max(T_0, T_1)$ such that $(X(T_2), Z(T_2)) \in V_{\epsilon/3}$.

Now we show that the trajectory remains in $V_\epsilon$ for all $t \geq T_2$. Indeed, by induction on $k$, we have $\bar L(X(T_2+kT_0), Z(T_2+kT_0)) \leq {2\epsilon/3}$ for all $k \in \Nbb$ (by~\eqref{eq:proof_1} and~\eqref{eq:proof_2}), then for all $h \in [0, T_0]$,
\al{
\bar L(X(T_2+kT_0+h), Z(T_2+kT_0+h)) 
&\leq \bar L(\Phi_h(X(T_2+kT_0), Z(T_2+kT_0))) + \epsilon/3 \\
&\leq 2\epsilon/3 + \epsilon/3.
}

Since $\epsilon$ is arbitrary, this proves that for all $\epsilon$, $(X(t), Z(t))$ remains in $V_\epsilon$ for $t$ large enough, a.s. But by definition of $\bar L$, $(x, z) \in V_\epsilon$ implies that $f(x) - f(x^\star) \leq \epsilon$, which proves that $f(X(t)) - f(x^\star)$ converges to $0$.
\end{proof}

\subsection{Convergence of expected function values}

Next, we derive explicit bounds on function values.
\begin{theorem}
\label{thm:rate_exp}
Suppose that Assumptions~\ref{assumption:lip} and \ref{assumption:volatility} hold. Suppose that $a = \eta / r$ and $\eta \geq \dot r$. Let $(X(t), Z(t))$ be the unique continuous solution to $\SAMD_{\eta, \eta/r, s}$. Then for all $t \geq t_0$,
\al{
\Ebb [f(X(t))] - f(x^\star) \leq \frac{L(x_0, z_0, t_0) + \psi(x^\star) (s(t)-s(t_0)) + \frac{n L_{\psi^*}}{2} \int_{t_0}^t \frac{\eta^2(\tau)\sigma_*^2(\tau)}{s(\tau)} d\tau}{r(t)}.
}
\end{theorem}
\begin{proof}
Integrating the bound of Lemma~\ref{lem:lyap_bound_sto}, and using the fact that $(f(X(t)) - f(x^\star))(\dot r - \eta) \leq 0$ by assumption on $\eta$, we have
\begin{equation}
\label{eq:lyap_bound_sto_int}
L(X(t), Z(t), t) - L(x_0, z_0, t_0) \leq \psi(x^\star) (s(t) - s(t_0)) + \frac{n L_{\psi^*}}{2} \int_{t_0}^t \frac{\eta^2(\tau)\sigma_*^2(\tau)}{s(\tau)} d\tau + \int_{t_0}^t \braket{V(\tau)}{dB(\tau)} ,
\end{equation}
Taking expectations, the last term vanishes since it is an It\^o martingale, and we conclude by observing that
$\Ebb[f(X(t))] - f(x^\star) \leq \Ebb [L(X(t), Z(t), t)] / r(t)$.
\end{proof}
To give a concrete example, suppose that $\sigma_*(t) = \Ocal(t^{\alpha_\sigma})$ is given, and let $r(t) = t^{\alpha_r}$ and $s(t) = t^{\alpha_s}$, $\alpha_r, \alpha_s > 0$. To simplify, we will take $\eta(t) = \dot r(t) = \alpha_r t^{\alpha_r - 1}$. Then the bound of Theorem~\ref{thm:rate_exp} shows that
$\Ebb[f(X(t))] - f(x^\star) = \Ocal(t^{\alpha_s - \alpha_r} + t^{\alpha_r + 2\alpha_\sigma - \alpha_s - 1})$. To minimize the asymptotic rate, we can choose $\alpha_s - \alpha_r = \alpha_r + 2\alpha_\sigma - \alpha_s - 1$, i.e. $\alpha_r + \alpha_\sigma - \alpha_s - \frac{1}{2} = 0$ (it is always possible to find such $\alpha_r, \alpha_s > 0$), and the resulting rate is $\Ocal(t^{\alpha_\sigma - \frac{1}{2}})$.
\begin{corollary}
\label{cor:sto_rate}
Suppose that Assumptions~\ref{assumption:lip} and \ref{assumption:volatility} hold. Suppose that $\sigma_*(t) = \Ocal(t^{\alpha_\sigma})$, $\alpha_\sigma < \frac{1}{2}$. Consider $\SAMD_{\eta, \eta/r, s}$ dynamics with $r(t) = t^{\alpha_r}$, $\eta(t) = \dot r(t) = \alpha_r t^{\alpha_r - 1}$, and $s(t) = t^{\alpha_s}$, and suppose that
\[
\alpha_r = \alpha_s - \alpha_\sigma + \frac{1}{2}
\]
Then
\al{
\Ebb[f(X(t))] - f(x^\star) = \Ocal(t^{\alpha_\sigma - \frac{1}{2}})  && \text{as $t \to \infty$.}
}
\end{corollary}
In particular, Corollary~\ref{cor:sto_rate} indicates that it is possible to obtain convergence rates faster than $\Ocal(\frac{1}{t})$ if the gradient noise decays fast enough.

\begin{remark}
This result is reminiscent of the results of~\cite{schmidt2011inexact}, who studied the convergence rates of inexact accelerated methods (in discrete time), in which the gradient is evaluated at step $k$ up to an error term $e_k$. They show that the optimality gap (function values in expectation) at step $k$ is $\Ocal(\frac{\sum_{\kappa = 1}^k \kappa e_\kappa}{k^2})$. Although the settings are different (their analysis is done for deterministic, discrete methods), the results are similar. In particular, when $e_k$ decays as $\Ocal(k^{\alpha_\sigma})$, the resulting convergence rate of Nesterov's accelerated method is $\Ocal(\frac{\sum_{\kappa = 1}^k\kappa \kappa^{\alpha_\sigma}}{k^2}) = \Ocal(k^{\alpha_\sigma})$, which appears slower than the $\Ocal(t^{\alpha_\sigma} - \frac{1}{2})$ of Corollary~\ref{cor:sto_rate}. This is due to the fact that Nesterov's method corresponds to $r(t) = \Theta(t^2), s(t) = \Theta(1)$ (see Appendix~\ref{sec:app:nesterov_ode}), which is not the optimal choice of decay rate according to the corollary.
\end{remark}

\begin{remark}[Dependence on the dimension]
Note that we can potentially scale the rates $\eta(t), r(t)$ in order to improve the dependence of the bound~\eqref{eq:lyap_bound_sto_int} on the dimension $n$. Note that the second term in the bound of Theorem~\ref{thm:rate_exp} involves the quantity $\psi(x^\star)$, which can be bounded by $\sup_{x \in \Xcal} \psi(x)$ when $\Xcal$ is compact, and this supremum is often a non-decreasing function of $n$ (for example, in the case of the negative entropy on the simplex, we have that $\sup_{x \in \Xcal} \psi(x) = \sup_{x \in \Xcal} \sum_{i = 1}^n x_i \ln x_i = \log n$). Thus let us denote the supremum by $M(n)$, and assume that $\eta, r, a$ are given and satisfy the conditions of Theorem~\ref{thm:rate_exp}. Define the rescaled weights $\bar \eta(t) = \eta_0(n) \eta(t)$, $\bar r(t) = \eta_0(n) r(t)$ (note that we scale both $\eta$ and $r$ so that the condition $\eta \geq \dot r$ still holds, and we do not need to scale $a$ since $a = \eta/r = \bar\eta / \bar r$). Then applying Theorem~\ref{thm:rate_exp} to $\SAMD_{\bar \eta, \bar \eta / \bar r, s}$ gives us
\al{
\Ebb &[f(X(t))] - f(x^\star)
\leq \frac{L(x_0, z_0, t_0) + M(n) s(t) + \frac{n L_{\psi^*} \eta_0^2(n)}{2} \int_{t_0}^t \frac{\eta^2(\tau)\sigma_*^2(\tau)}{s(\tau)} d\tau}{\eta_0(n)r(t)}
}
where $L(x_0, z_0, t_0) = \eta_0(n)r(t_0)(f(x_0) - f(x^\star)) + s(t_0) D_{\psi^*}(z_0, z^\star)$ and is dominated by the other terms (assuming that $D_{\psi^*}(z_0, z^\star) = \Ocal(M(n))$ to simplify). Thus we have, asymptotically as $n \to \infty$,
\al{
\Ebb &[f(X(t))] - f(x^\star)
= \Ocal\parenth{ \frac{M(n)}{\eta_0(n)} + n \eta_0(n) }
}
and choosing $\eta_0(n) = \sqrt{\frac{M(n)}{n}}$ minimizes the asymptotic growth rate of the bound, resulting in
\al{
\Ebb [f(X(t))] - f(x^\star)
= \Ocal( \sqrt{nM(n)} ) && \text{as $n \to \infty$}.
}
\end{remark}

We now illustrate the asymptotic rates of Theorem~\ref{thm:rate_exp}, and the optimal choice of primal and dual rates given in Corollary~\ref{cor:sto_rate} on the simplex-constrained sum-exponential example of Section~\ref{sec:example}. We simulate solution trajectories of $\SAMD_{\eta, \eta/r, s}$ under different noise regimes, by taking $\sigma^*(t) = 10^{-1} t^{\alpha_\sigma}$ for different values of $\alpha_\sigma$, and for different configurations of rates $r(t) = t^{\alpha_r}$, $s(t) = t^{\alpha_s}$. To simplify, we took $\eta(t) = \dot r(t)$. We plot the mean and standard deviation over $100$ simulations. In particular, we seek to verify whether the optimal decay rates are given by $\alpha_r = \alpha_s - \alpha_\sigma + \frac{1}{2}$, as predicted by Corollary~\ref{cor:sto_rate}. The results are given in Figure~\ref{fig:samd_rates}.

For each choice of $\alpha_r$, we evaluate the decay rate of expected function values (using the empirical mean over $100$ runs). The optimal decay rate $\alpha_r$, and the resulting decay of function values, seem consistent with the predictions of the Corollary when $\sigma_*(t)$ is non-decreasing (i.e. $\alpha_\sigma \geq 0$). However, for decreasing $\sigma_*(t)$ (i.e. in the vanishing noise regime), the estimates of the Corollary seem to be conservative. The optimal rate $\alpha_r$ which we observed was slower than predicted, and the resulting convergence rate was faster, see Figure~\ref{fig:samd_rates} for examples. It is also interesting to observe the effect of varying the averaging rates on the trajectory: a faster decay of the sensitivity $1/s(t)$, or similarly a faster decay of the dual weight $\eta(t)$, results in a decrease of the period of oscillations of the accelerated dynamics.

\def\wdth{.51\textwidth}
\begin{figure}[h!]
\centering
\includegraphics[width=\wdth]{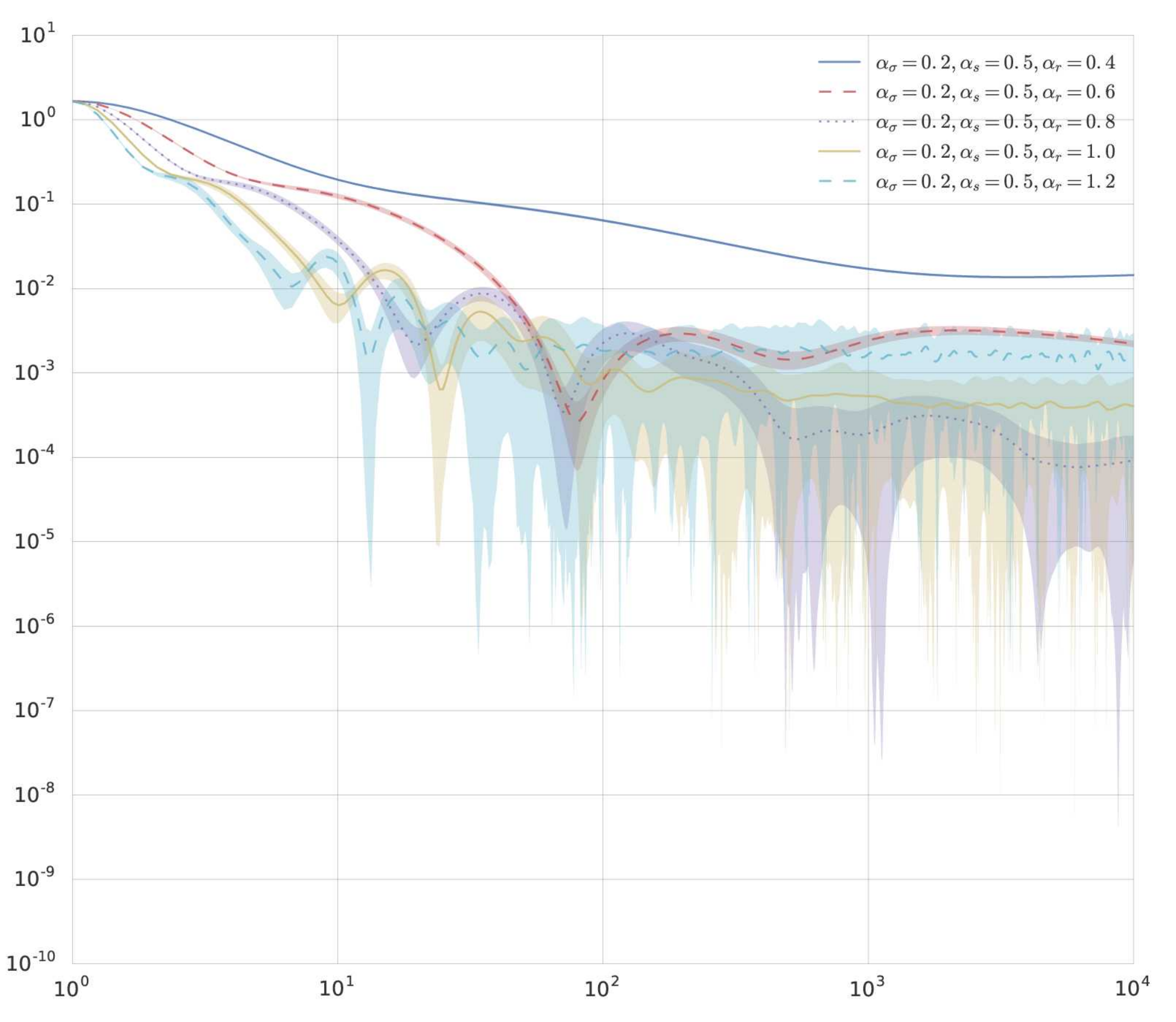}%
\includegraphics[width=\wdth]{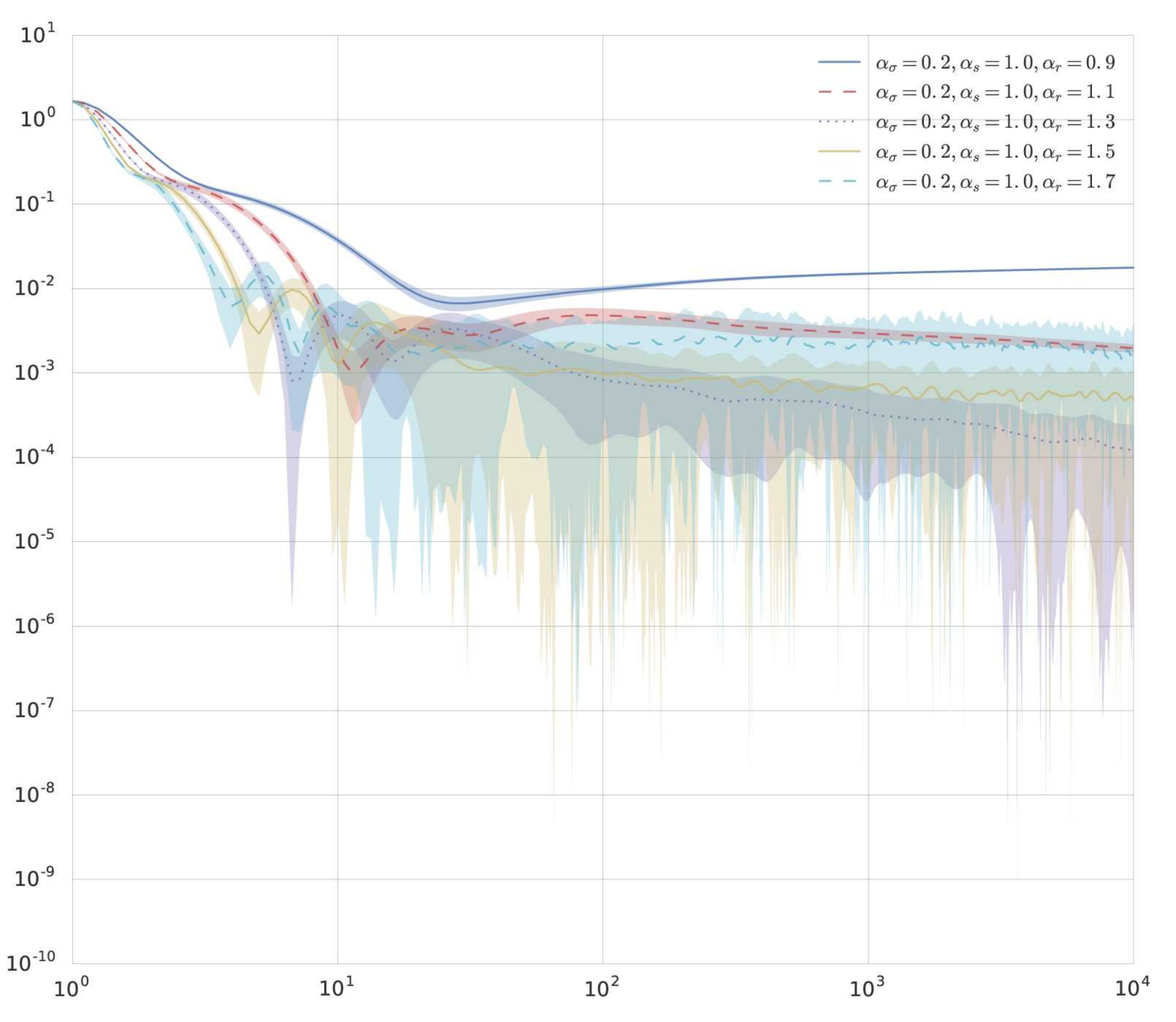}%

\includegraphics[width=\wdth]{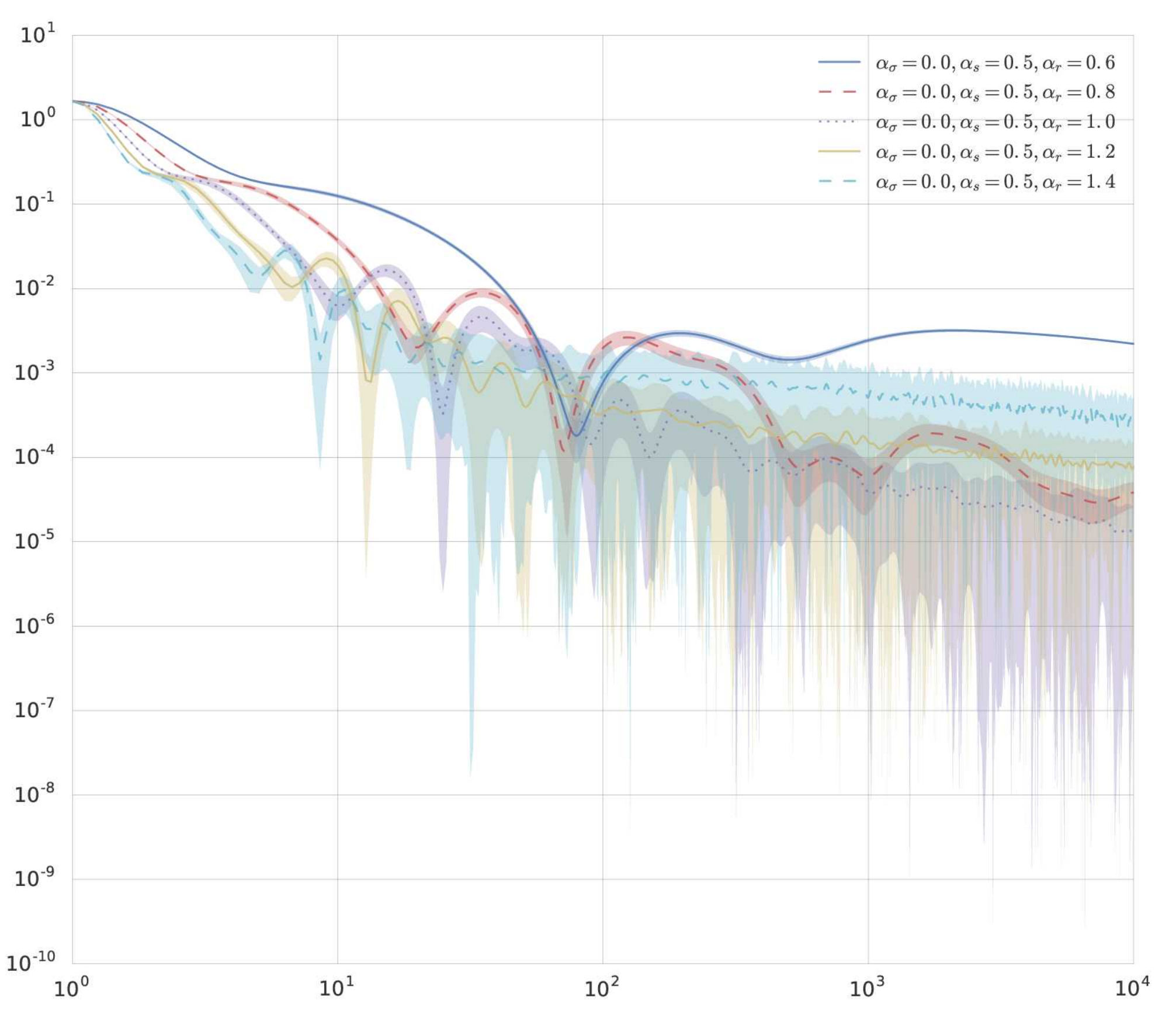}%
\includegraphics[width=\wdth]{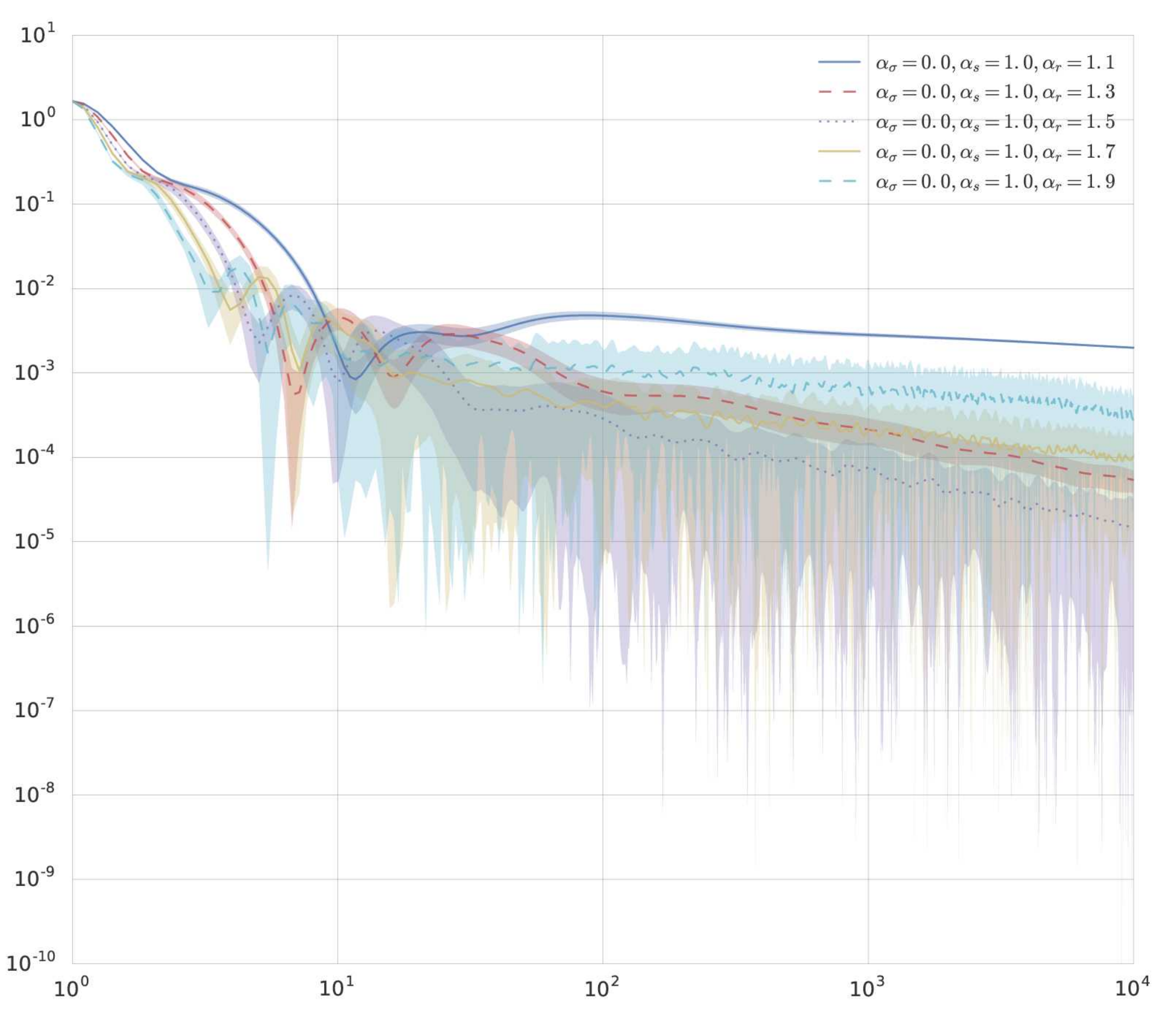}%

\includegraphics[width=\wdth]{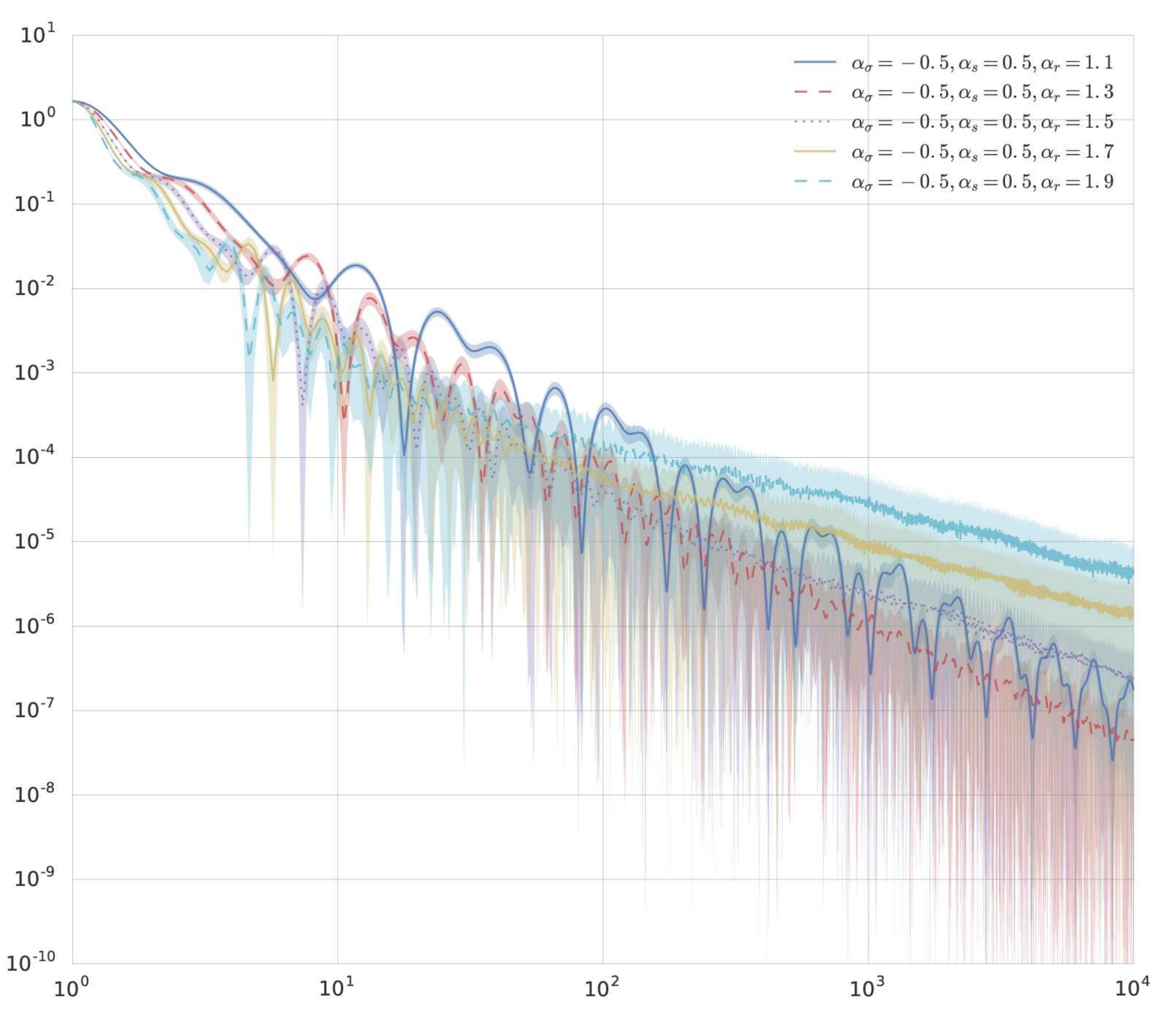}%
\includegraphics[width=\wdth]{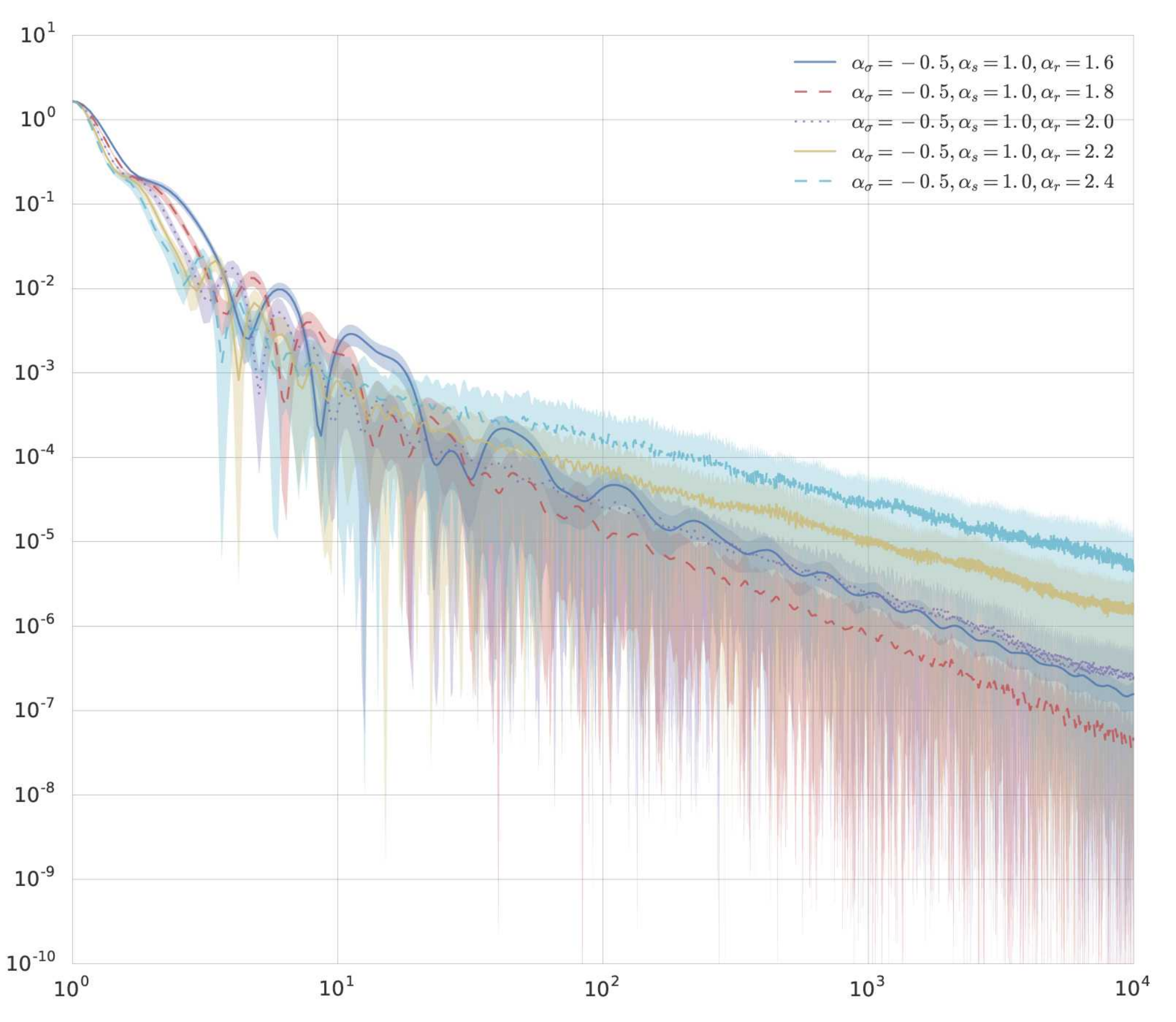}%

\end{figure}
\clearpage
\begin{figure}[h!]\ContinuedFloat
\includegraphics[width=\wdth]{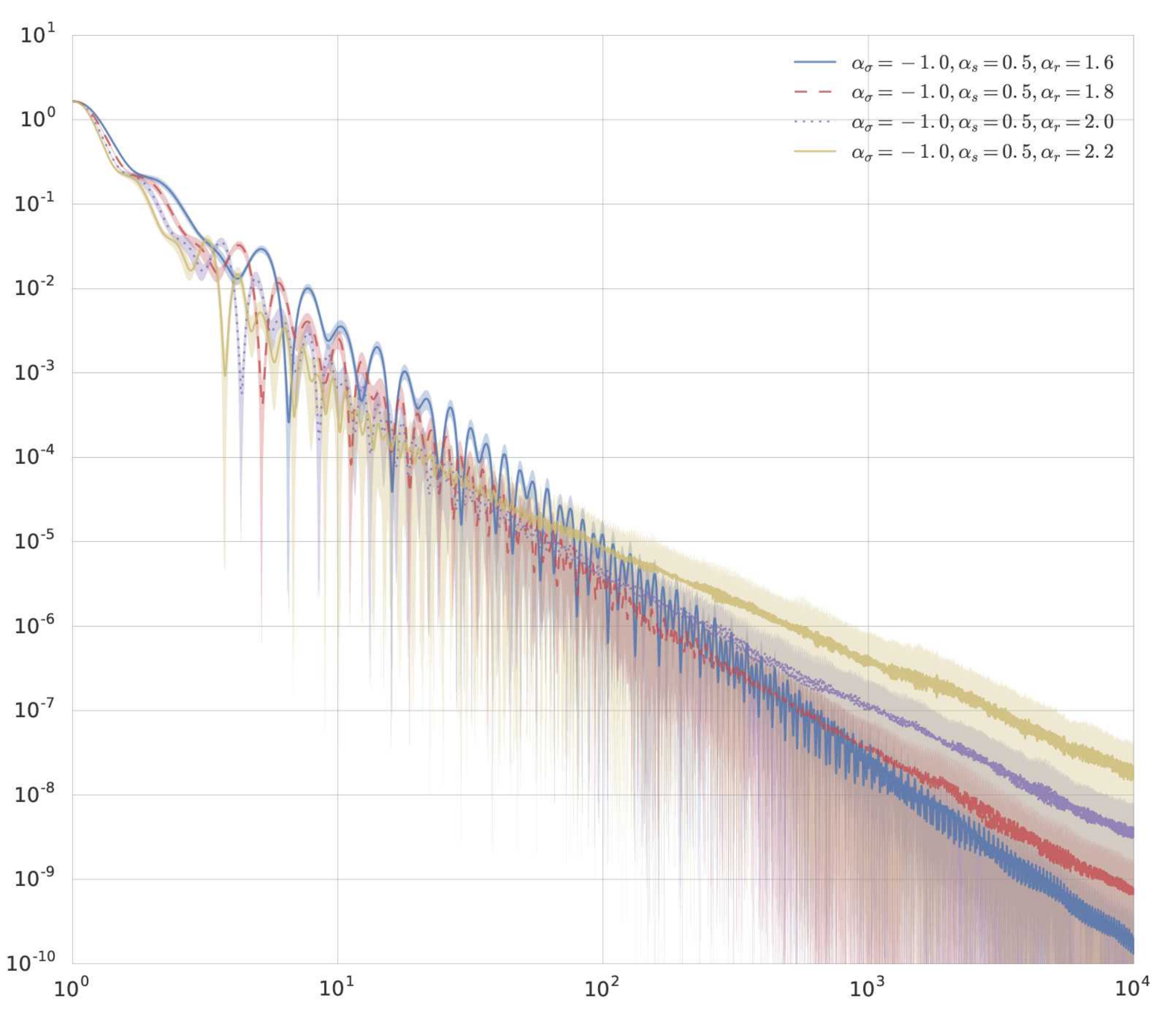}%
\includegraphics[width=\wdth]{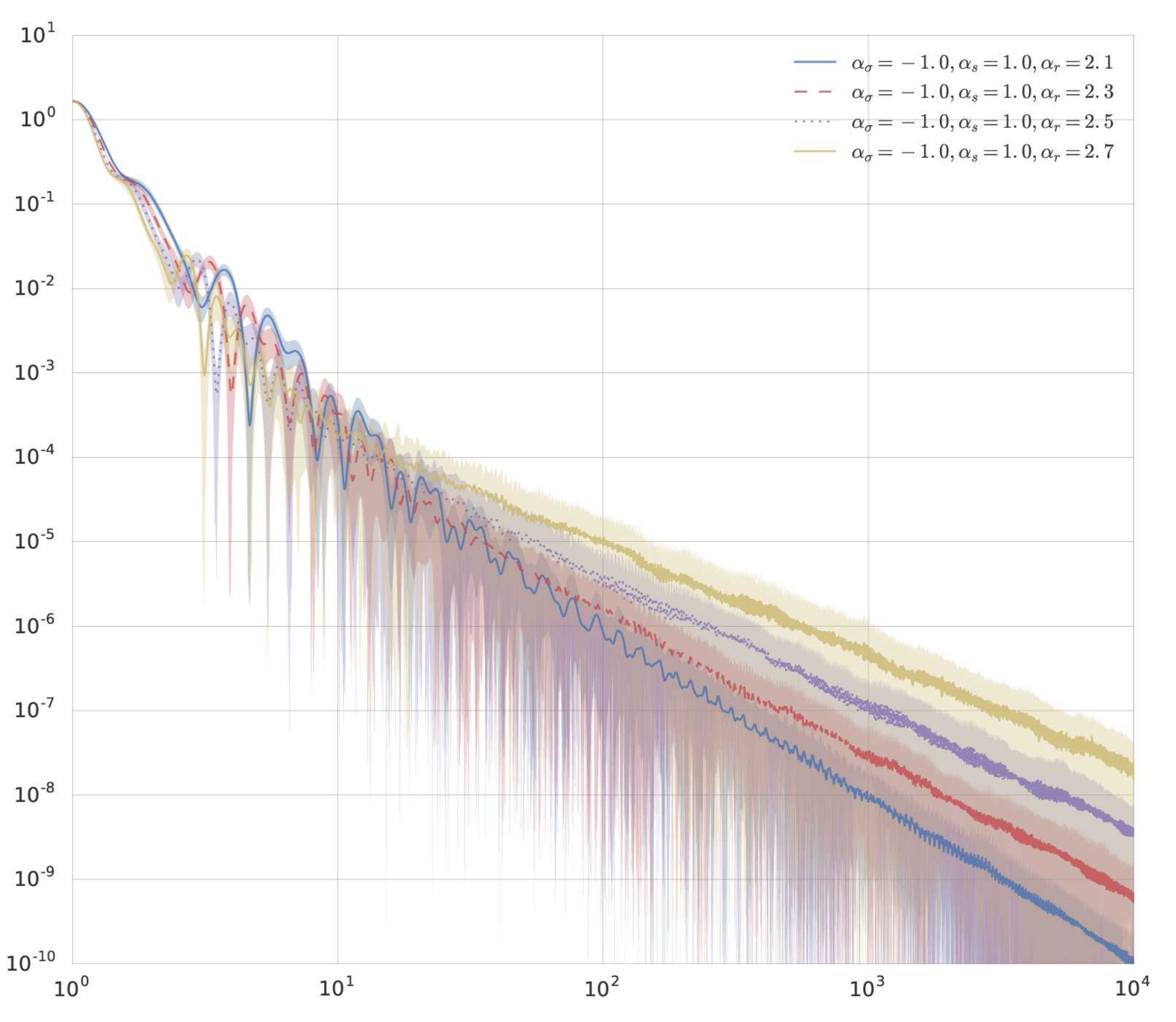}%
\caption{Mean and standard deviation of function values along solution trajectories of $\SAMD_{\dot r, \dot r/r, s}$, for different configurations of rates $r(t) = t^{\alpha_r}, s(t) = t^{\alpha_s}$, and different noise regimes $\sigma_*(t) = 10^{-1} t^{\alpha_\sigma}$. Each row correspond to a different noise regime ($\alpha_\sigma = .2$ for the first row, $0$ for the second, $-.5$ for the third, and $-1$ for the last), and each column corresponds to a different decay rate of the sensitivity ($\alpha_s = .5$ for the first column, $1$ for the second). Each figure contains multiple plots corresponding to different $\alpha_r$. For the first two rows, $\alpha_r = \alpha_s - \alpha_\sigma + .5$ gives the optimal rate, as predicted by Corollary~\ref{cor:sto_rate}. However, for the last two rows, the best decay rate is achieved by $\alpha_r$ is lower, namely $\alpha_r = \alpha_s - \alpha_\sigma + .3$ for the third row, and $\alpha_r = \alpha_s - \alpha_\sigma + .1$ for the last.}
\label{fig:samd_rates}
\end{figure}

\subsection{Almost sure asymptotic rates}
Finally, we give an estimate of the asymptotic convergence rate along solution trajectories, under the additional compactness assumption (to be able to bound the martingale term, using Lemma~\ref{lem:martingale_bound}).
\begin{theorem}
\label{thm:rate_as}
Suppose that Assumptions~\ref{assumption:lip}, \ref{assumption:compact} and \ref{assumption:volatility} hold, and suppose that $a = \eta / r$ and $\eta \geq \dot r$. Let $(X(t), Z(t))$ be the unique continuous solution to $\SAMD_{\eta, \eta/r, s}$. Then
\al{
f(X(t)) - f(x^\star) = \Ocal\parenth{\frac{s(t) + n\int_{t_0}^t \frac{\eta^2(\tau)\sigma_*^2(\tau)}{s(\tau)} + \sqrt{b(t) \log \log b(t)}}{r(t)}}  &&\text{a.s. as $t \to \infty$,}
}
where $b(t) = \int_{t_0}^t \eta^2(\tau)\sigma_*^2(\tau)d\tau$.
\end{theorem}
\begin{proof}
Integrating the bound of Lemma~\ref{lem:lyap_bound_sto} once again, we get inequality~\eqref{eq:lyap_bound_sto_int}, where we can bound the It\^o martingale term $\int_{t_0}^t \braket{V(\tau)}{dB(\tau)}$ using Lemma~\ref{lem:martingale_bound}, since $\Xcal$ is now assumed compact. This concludes the proof.
\end{proof}
Comparing the last bound to that of Theorem~\ref{thm:rate_exp}, we have the additional $\frac{\sqrt{b(t) \log\log b(t)}}{r(t)}$ term due to the envelope of the martingale term. This results in a slower a.s. convergence rate. To give a concrete example, suppose again that $\sigma_*(t) = \Ocal(t^\alpha_\sigma)$, and that $r(t) = t^\beta$ and $\eta(t) = \dot r(t) = \beta t^{\beta - 1}$ to simplify. Then $b(t) = \int_{t_0}^t \eta^2(\tau)\sigma_*^2(\tau)d\tau = \Ocal(t^{2\beta+2\alpha_\sigma - 1})$, and the martingale term becomes $\Ocal(\sqrt{b(t) \log \log b(t)} / r(t)) = \Ocal(t^{\alpha_\sigma - \frac{1}{2}} \sqrt{\log \log t})$. Remarkably, the convergence rate of sample trajectories is, up to a $\sqrt{\log \log t}$ factor, the same as the convergence rate in expectation.
\begin{corollary}
Suppose that Assumptions~\ref{assumption:lip}, \ref{assumption:compact} and \ref{assumption:volatility} hold, and suppose that $\sigma_*(t) = \Ocal(t^{\alpha_\sigma})$, $\alpha_\sigma < \frac{1}{2}$. Consider $\SAMD_{\eta, \eta/r, s}$ dynamics with $r(t) = t^{\alpha_r}$, $\eta(t) = \dot r(t) = \alpha_r t^{\alpha_r - 1}$, and $s(t) = t^{\alpha_s}$, and suppose that
\[
\alpha_r = \alpha_s - \alpha_\sigma + \frac{1}{2}.
\]
Then
\al{
f(X(t)) - f(x^\star) = \Ocal(t^{\alpha - \frac{1}{2}} \sqrt{\log \log t}) &&\text{a.s. as $t \to \infty$.}
}
\end{corollary}

\section{Numerical examples}
\label{sec:numerics}
In this section, we give additional numerical examples to illustrate the differences between $\SMD$ and $\SAMD$ dynamics. In order to compare the performance of the two methods, we derive, in Appendix~\ref{sec:app:smd_rates}, the convergence rates of (non-accelerated) stochastic mirror descent, in particular when $\sigma_*(t)$ is asymptotically vanishing. Corollary~\ref{cor:smd_rate} shows that when $\sigma_*(t) = \Ocal(t^{\alpha_\sigma})$, $\SMD$ with inverse sensitivity rate $s(t) = t^{\max(0, \alpha_\sigma + \frac{1}{2})}$ guarantees convergence of expected function values at the rate $\Ocal(t^{\max(\alpha_\sigma - \frac{1}{2}, 1)})$. So theoretically, when $\alpha_\sigma \in [-\frac{1}{2}, \frac{1}{2})$, both methods are able to achieve the same asymptotic rates. However in practice, the accelerated dynamics exhibit desirable properties that we attempt to illustrate.

We consider again the simplex-constrained example of Section~\ref{sec:example}, so that we can visualize the trajectories and provide intuition. The objective function is given by the sum-exp function
\[
f_1(x) = \sum_{i = 1}^k e^{\braket{c_i}{x}},
\]
and in order to illustrate the non-strongly convex case, we also consider a simple convex quadratic of rank $1$, given by
\[
f_2(x) = \frac{1}{2}\braket{x}{c}^2 = \frac{1}{2} x^T cc^Tx.
\]

The results are given in Figure~\ref{fig:comparison_sum_exp} and~\ref{fig:comparison_quad_over_lin}, respectively. Qualitatively, $\SAMD$ appears to exhibit smoother trajectories and a lower variance of the function values. Due to the oscillations of the accelerated dynamics, $\SAMD$ initially makes slower progress (observe how $\SMD$ trajectories make very fast progress in the initial steps), but makes faster progress in later iterations. A similar behavior is typical in the deterministic counterpart of the dynamics: accelerated mirror descent typically exhibits slower progress initially compared to non-accelerated mirror descent, but eventually makes faster progress.

Another difference which we observed in our numerical experiments is that the $\SAMD$ dynamics seems more robust to much larger discretization step sizes (where we used a simple, constant-step discretization). Thus, even when $\SMD$ theoretically achieves the same asymptotic rates of convergences as $\SAMD$ (this is the case when $\sigma_*(t) = \Theta(t^{\alpha_\sigma})$ with $\alpha_\sigma \in [-\frac{1}{2}, \frac{1}{2})$, according to Corollary~\ref{cor:sto_rate} and Corollary~\ref{cor:smd_rate}), $\SAMD$ can, in practice, reach a target accuracy in fewer iterations.

In order to further illustrate the effect of the magnitude of the noise, we generate, in Figure~\ref{fig:noise}, different trajectories corresponding to $\SMD$ and $\SAMD$ dynamics, with a constant noise covariance given by $\sigma_*(t) \equiv \sigma_*$. It is interesting to observe that for $\SAMD$, the magnitude of the noise seems to affect the amplitude of the oscillations, but not their period (the amplitude becomes smaller as the magnitude of the noise increases).

\def\wdth{\textwidth}
\begin{figure}[h!]
\centering
\includegraphics[width=\wdth]{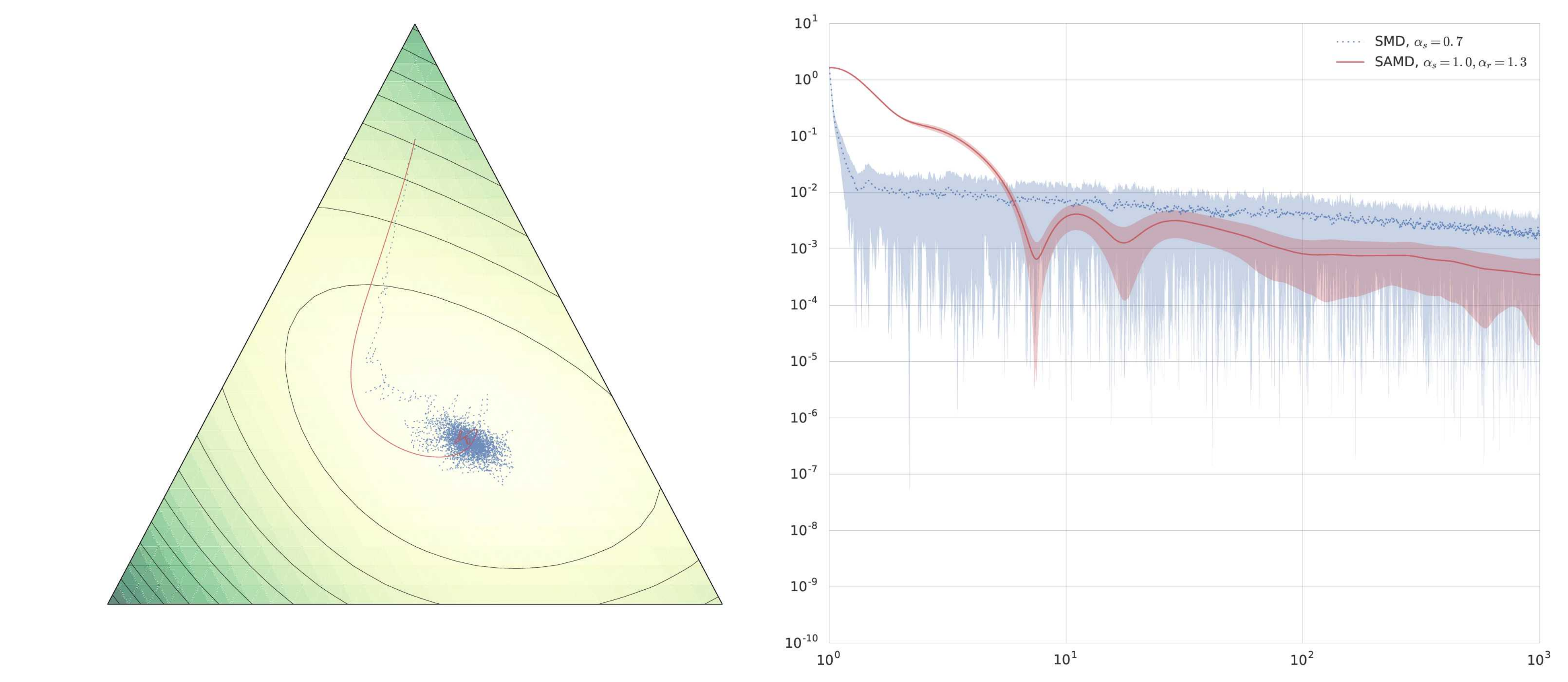}
\includegraphics[width=\wdth]{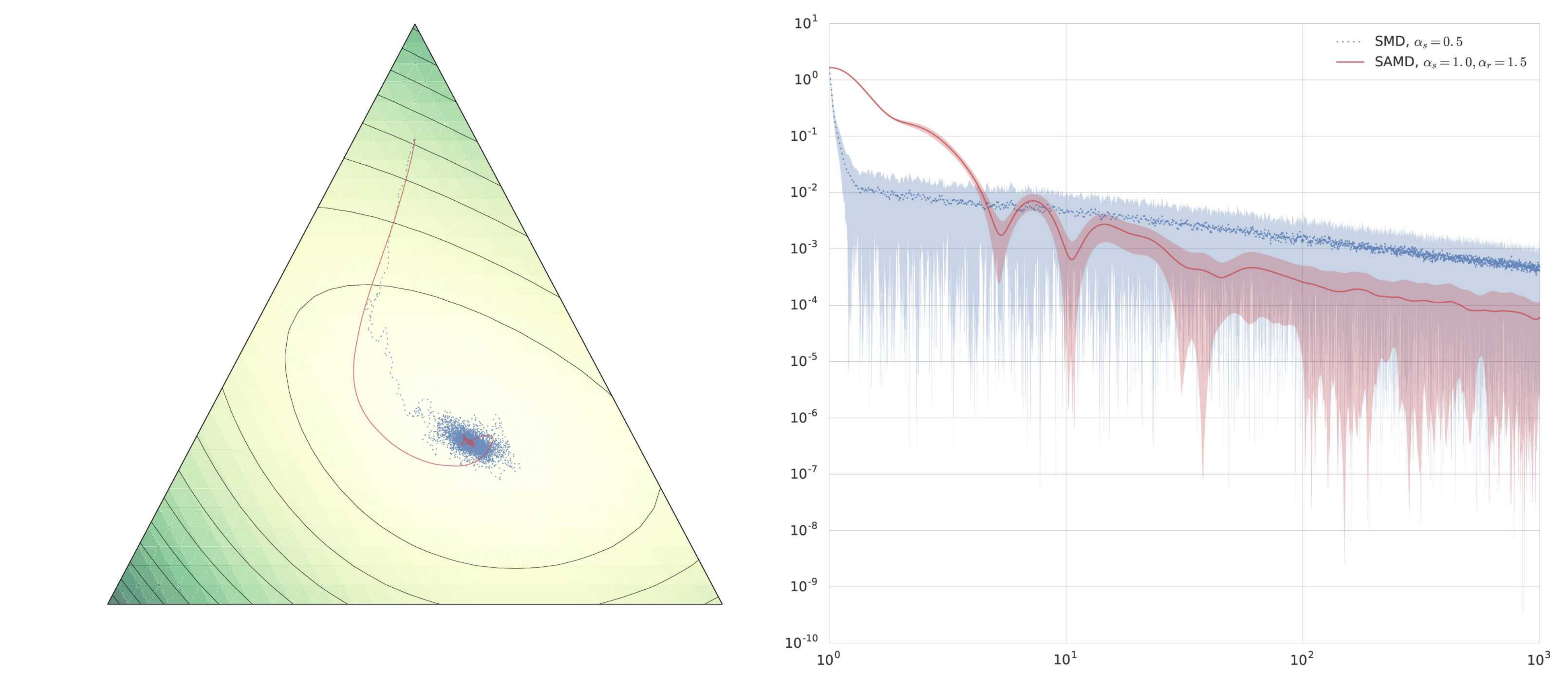}
\includegraphics[width=\wdth]{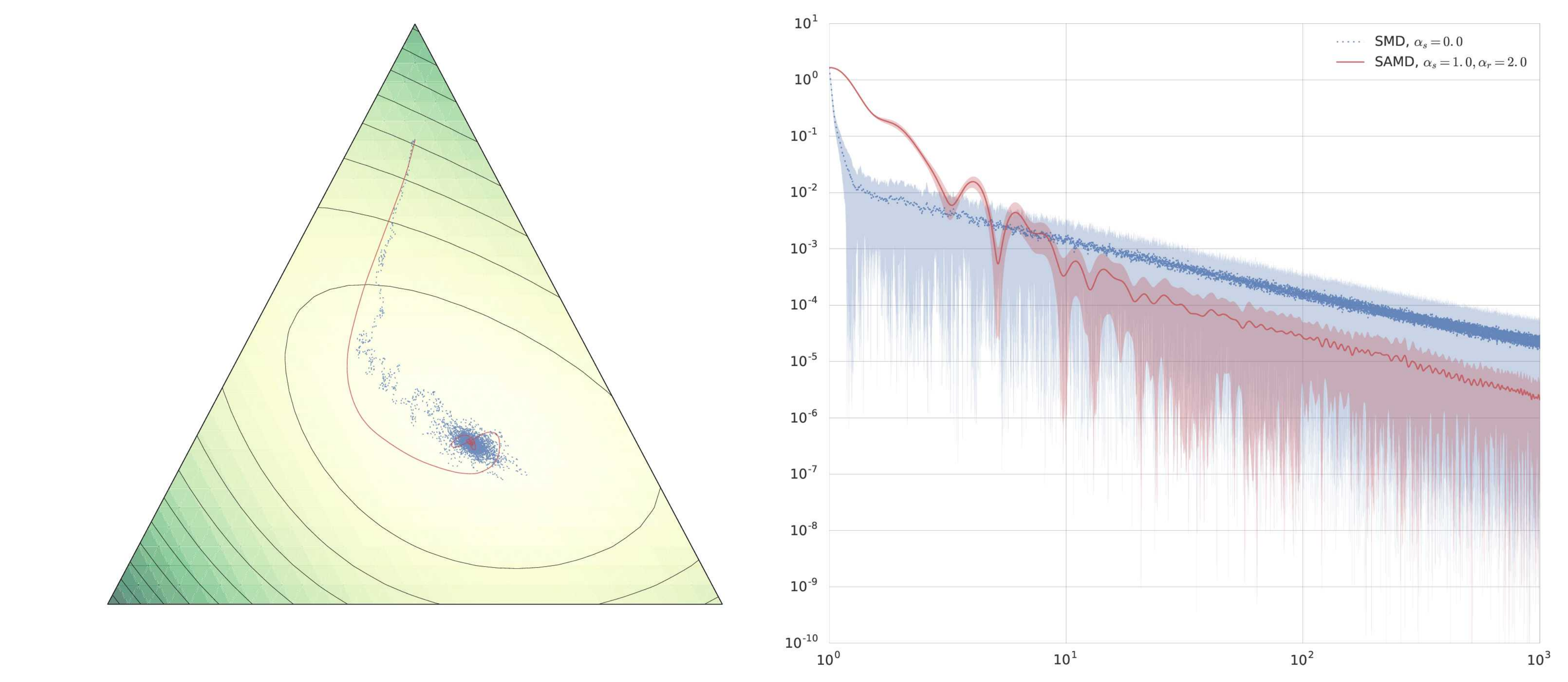}

\caption{Function values $f_1(X(t)) - f_1(x^\star)$, along solution trajectories $X(t)$ of $\SMD$ and $\SAMD$, for different noise regimes $\sigma_*(t) = 10^{-1} t^{\alpha_\sigma}$: $\alpha_\sigma = .2$ for the top figure, $\alpha_\sigma = 0$ for the middle figure, and $\alpha_\sigma = -\frac{1}{2}$ for the bottom figure. The dynamics are configured according to the optimal rates of Corollary~\ref{cor:sto_rate} and Corollary~\ref{cor:smd_rate}, i.e. $\alpha_s = \max(0, \alpha_\sigma + \frac{1}{2})$ for $\SMD$, and $\alpha_r = \alpha_s - \alpha_\sigma + \frac{1}{2}$ for $\SAMD$.}
\label{fig:comparison_sum_exp}
\end{figure}
\clearpage
\begin{figure}[h!]
\centering
\includegraphics[width=\wdth]{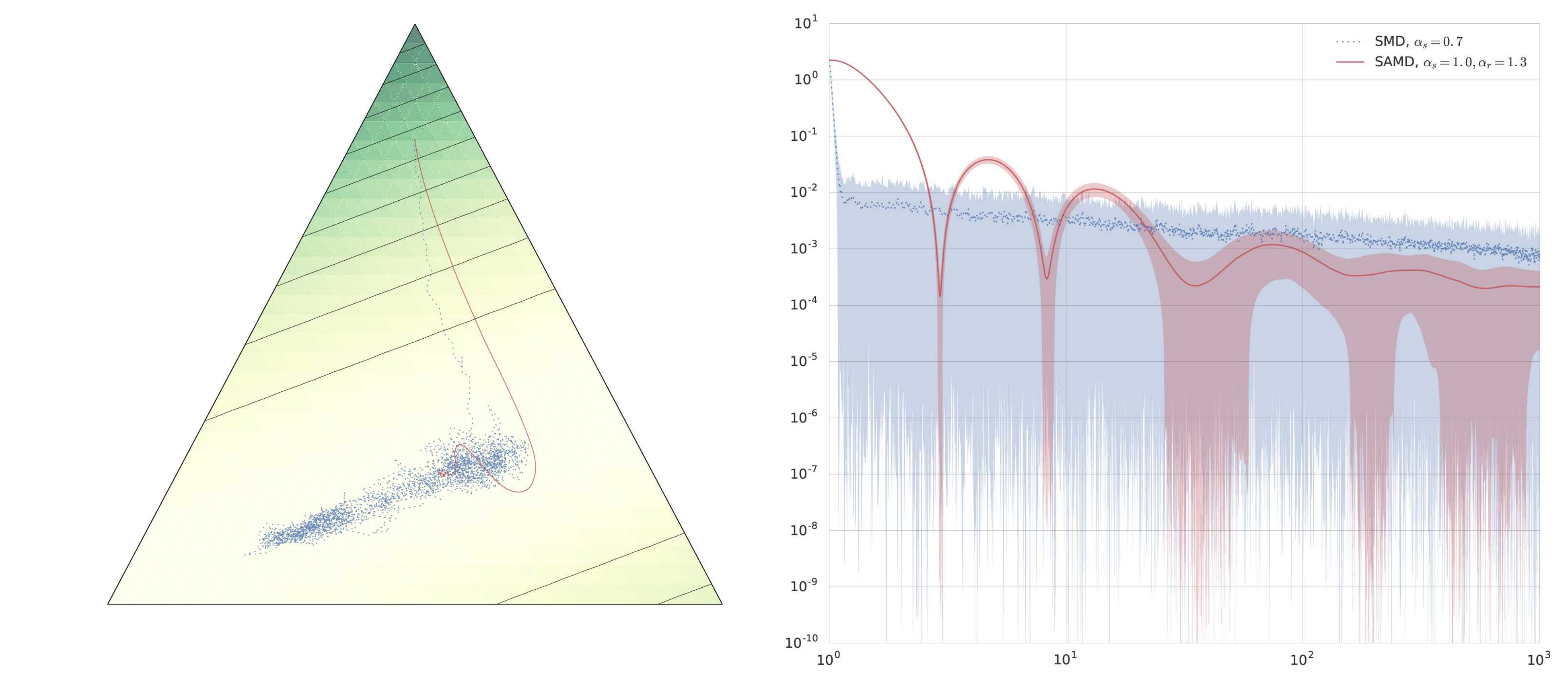}
\includegraphics[width=\wdth]{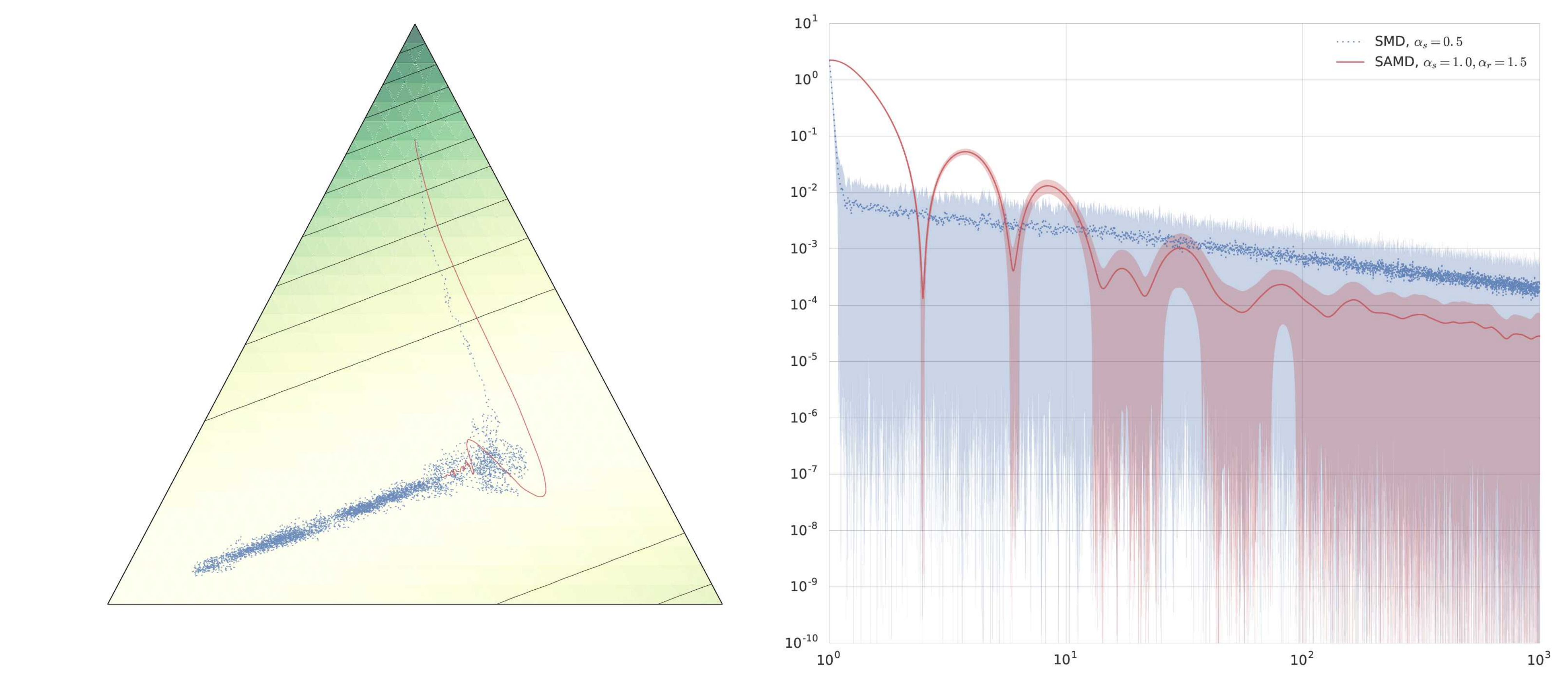}
\includegraphics[width=\wdth]{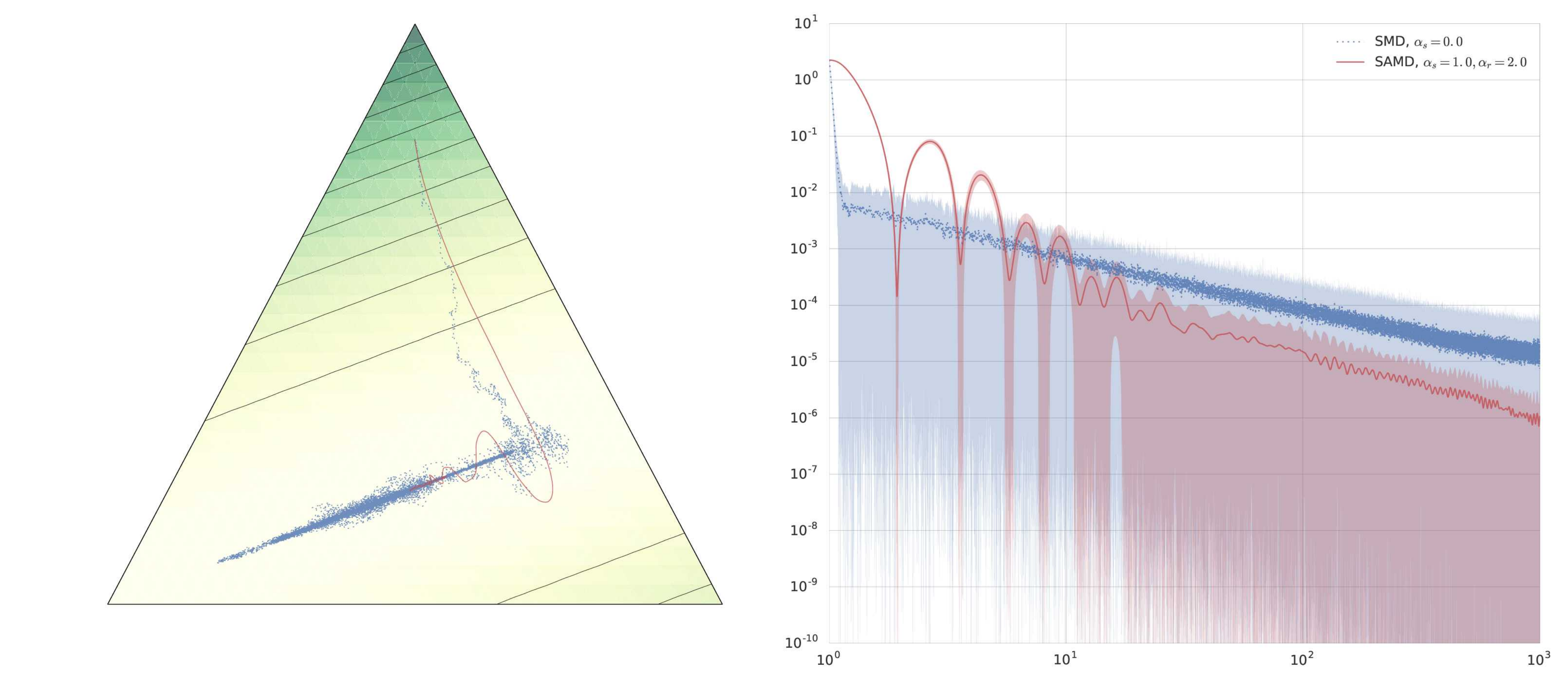}

\caption{Function values $f_2(X(t)) - f_2(x^\star)$, along solution trajectories $X(t)$ of $\SMD$ and $\SAMD$, for different noise regimes $\sigma_*(t) = 10^{-1} t^{\alpha_\sigma}$.}
\label{fig:comparison_quad_over_lin}
\end{figure}
\clearpage

\begin{figure}[h]
\centering
\includegraphics[width=1.\textwidth]{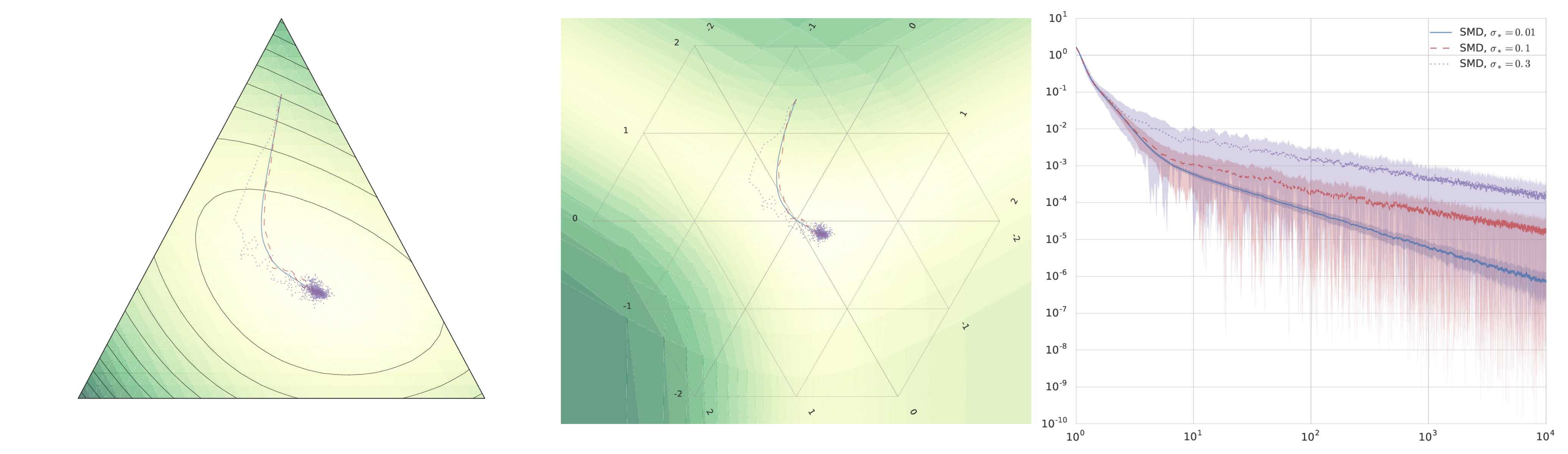}
\includegraphics[width=1.\textwidth]{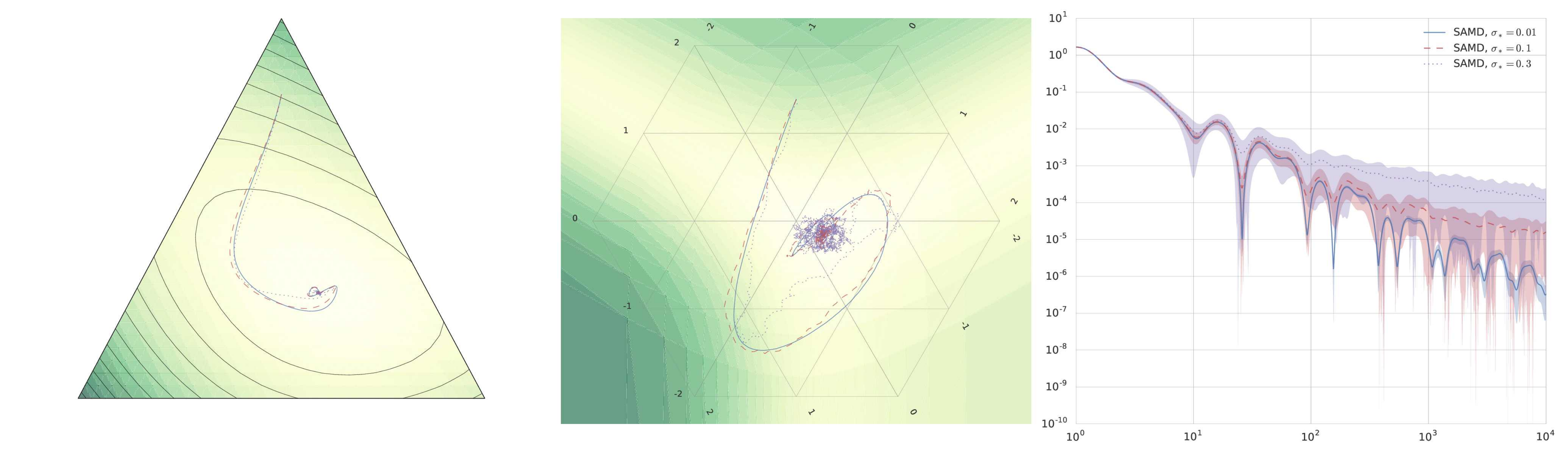}
\caption{Solution trajectories of $\SMD$ (top) and $\SAMD$ (bottom) dynamics for different values of noise covariance $\sigma_*$.}
\label{fig:noise}
\end{figure}

\section{Discussion}
\label{sec:discussion}

Starting from the averaging formulation of accelerated mirror descent in continuous-time, and motivated by stochastic optimization, we formulated a stochastic variant by adding a noise process to the gradient, and studied the resulting SDE. We discussed the role played by each parameter: the dual learning rate $\eta(t)$, the inverse sensitivity parameter $s(t)$, and the noise covariation bound $\sigma_*(t)$. In particular, we showed how the asymptotic bounds on $\sigma_*(t)$ affect the choice of $\eta(t), s(t)$.

Our results show that in the persistent noise regime, thanks to averaging, it is possible to guarantee a.s. convergence, remarkably even when $\sigma_*(t)$ is increasing (as long as $\sigma_*(t) = \Ocal(\sqrt t)$). In the vanishing noise regime, the appropriate choice of $\eta(t), s(t)$ leads to improved convergence rates, e.g. to $\Ocal(t^{\alpha_\sigma - \frac{1}{2}})$ in expectation and $\Ocal(t^{\alpha_\sigma - \frac{1}{2}} \sqrt{\log \log t})$ almost surely, when $\sigma_*(t) = \Ocal(t^{\alpha_\sigma})$ for negative $\alpha_\sigma$. These asymptotic bounds in continuous-time can provide guidelines in setting the different parameters of accelerated stochastic mirror descent.

\subsection*{Effects of time-change}
It is also worth observing that in the deterministic case, one can theoretically obtain arbitrarily fast convergence, through a time change as observed by~\cite{wibisono2016variational} -- a time-change would simply result in using different weights $\eta(t)$ and $a(t)$. In the stochastic dynamics, such a time-change would also lead to re-scaling the noise co-variation, and does not lead to a faster rate. To some extent, adding the noise prevents us from ``artificially'' accelerating convergence using a simple time-change. To illustrate this difference, first consider a time-change in the deterministic case. Let $(x(t), z(t))$ be the unique solution to $\AMD_{\eta, a, 1}$ (where we took $s(t) \equiv 1$ to simplify), and consider a differentiable increasing function of time, $\gamma(t)$. Let $(x', z')$ be defined by $x'(t) = x(\gamma(t))$ and $z'(t) = z(\gamma(t))$. Then $(x', z')$ satisfy the following dynamics:
\[
\begin{cases}
\dot z'(t) = - \dot \gamma(t)\eta(\gamma(t)) \nabla f(x'(t)) \\
\dot x'(t) = \dot\gamma(t)a(\gamma(t)) [\nabla \psi^*(z'(t)) - x'(t)],
\end{cases}
\]
thus $(x', z')$ is the unique solution to $\AMD_{\tilde \eta, \tilde a, 1}$ where $\tilde \eta(t) = \dot \gamma(t)\eta(\gamma(t))$ and $\tilde a(t) = \dot \gamma(t) a(\gamma(t))$, and if $\gamma$ is super linear, $f(x'(t))$ will have a faster convergence rate than $f(x(t))$. Indeed, if $\eta, a, r$ satisfy the conditions of Corollary~\ref{cor:deterministic_rate} (i.e. $a = \eta/r$ and $\eta \geq \dot r$), then $f(x(t)) - f(x^\star) \leq \frac{L(x_0, z_0, t_0)}{r(t)}$; but $\tilde \eta = \dot \gamma \eta \circ \gamma, \tilde a = \dot \gamma a \circ \gamma, \tilde r = r \circ \gamma$ also satisfy the conditions of the corollary, thus
\[
f(x'(t)) - f(x^\star) \leq \frac{L(x_0, z_0, t_0)}{r(\gamma(t))}.
\]

Let us now consider a similar time-change in the stochastic case. Let $(X, Z)$ be the unique (a.s.) continuous solution of $\SAMD_{\eta, \eta/r, s}$, and define $(X', Z')$ by the (differentiable, increasing) time-change $X'(t) = X(\gamma(t))$ and $Z'(t) = Z(\gamma(t))$. Then using the following time-change identity for It\^o martingales (see, e.g. Lemma 2.3 in~\citep{kobayashi2011stochastic}): $\int_{\gamma(t_0)}^{\gamma(t)} \sigma(X(\tau), \tau) dB(\tau) = \int_{t_0}^t \sigma(X(\gamma(\tau)), \gamma(\tau)) dB(\gamma(\tau))$, we have
\[
\begin{cases}
dZ'(t) = -\eta(\gamma(t)) [\nabla f(X'(t)) \dot \gamma(t)dt + \sigma(X'(t), \gamma(t)) dB(\gamma(t))] \\
dX'(t) = a(\gamma(t)) [\nabla \psi^*(Z'(t)/s(\gamma(t))) - X'(t)] \dot \gamma(t) dt,
\end{cases}
\]
which we can rewrite as
\[
\begin{cases}
dZ'(t) = -\tilde \eta(t) d\tilde G(t) \\
dX'(t) = \tilde a(t)[\nabla \psi^*(Z'(t) / \tilde s(t)) - X'(t)]dt.
\end{cases}
\]
where $\tilde \eta, \tilde a$ are as defined in the deterministic case, $\tilde s = s \circ \gamma$, and $\tilde G$ is defined by
\[
d\tilde G(t) = \nabla f(X'(t)) + \frac{\sigma(X'(t), \gamma(t))}{\dot \gamma(t)} dB(\gamma(t)).
\]
In particular, we observe that the noise covariation of $\tilde G$ is
\begin{equation}
\label{eq:quad_cov_Gtilde}
d[\tilde G_i(t), \tilde G_j(t)] = \frac{1}{\dot \gamma(t)^2} \Sigma_{ij}(X'(t), \gamma(t)) \dot \gamma(t) dt
\end{equation}
where we used the fact that the time-changed Brownian motion $B(\gamma(t))$ has quadratic covariation given by  $d[B_i(\gamma(t)), B_i(\gamma(t))] = \dot \gamma(t)dt$.

Comparing the quadratic covariation of $G$ and $\tilde G$ (equations~\eqref{eq:quad_cov_G} and~\eqref{eq:quad_cov_Gtilde} respectively), it becomes apparent that, unless $\gamma$ is the identity, rescaling time also rescales the covariation of the noise (even in the case where $\Sigma(x, t)$ does not depend on $t$, due to the $\dot \gamma(t)$ term). In other words, accelerating time by $\gamma(t)$ would scale down the variance of the gradient by $\dot \gamma(t)$, and $(X', Z')$ would not be a solution to the original problem anymore, unlike in the deterministic case.

\subsection*{Future directions}
Finally, we believe this continuous-time analysis can be extended in several interesting directions. For instance, it will be interesting to carry a similar analysis for other classes of convex problems, such as strongly convex functions, for which we expect faster optimal rates. In the deterministic case, many heuristics have been developed, which are known to empirically improve the convergence rate, such as the restarting heuristics of~\cite{ODonoghue2015adaptive}, the speed restarting of~\cite{su2014differential}, and the adaptive averaging heuristic of~\cite{krichene2016adaptive}. It will be interesting to adapt these heuristics to the stochastic case.

\section*{Acknowledgements}
We gratefully acknowledge the support of the NSF through grant IIS-1619362 and of the Australian Research Council through an Australian Laureate Fellowship (FL110100281) and through the Australian Research Council Centre of Excellence for Mathematical and Statistical Frontiers (ACEMS).


{\bibliographystyle{abbrvnat}
\bibliography{learning}}
\newpage

\appendix
\section{Dynamics of Nesterov's accelerated method}
\label{sec:app:nesterov_ode}
Nesterov's accelerated method~\citep{nesterov1983method} has been shown by~\cite{su2014differential} to be the discretization of the ODE: $\ddot x(t) = -\nabla f(x(t)) - \frac{\alpha}{t} \dot x(t)$ with $\alpha \geq 3$, which describes the motion of a damped non-linear oscillator, driven by the potential $f$, and subject to a viscous friction term $\frac{\alpha}{t} \dot X$. \cite{cabot2009dissipation} had previously studied a general family of such damped oscillators with vanishing friction, but the connection with Nesterov's method was not made until \cite{su2014differential}. Note that the dynamics are unconstrained in this case. This ODE can be recovered as a special case of $\AMD$ by taking the identity as a mirror map (which corresponds to taking $\psi(x) = \frac{1}{2}\|x\|^2_2$): Writing the second equation of $\AMD$ as $\frac{\dot x(t)}{a(t)} = z(t) - x(t)$ (where we took $s \equiv 1$) and taking derivatives, we have
\al{
\frac{1}{a(t)}\ddot x(t) - \frac{\dot a(t)}{a^2(t)} \dot x(t) = \dot z(t) - \dot x(t) = - \eta(t)\nabla f(x(t)) - \dot x(t),
}
i.e.
\[
\ddot x(t) = -\eta(t)a(t) \nabla f(x(t)) - \dot x(t) \frac{a^2(t) - \dot a(t)}{a(t)},
\]
and by taking $r(t) = \frac{t^2}{\beta^2}$, $\eta(t) = \frac{t}{\beta}$, $a(t) = \frac{\beta}{t}$, with $\beta \geq 2$, the ODE becomes
\[
\ddot x(t) = -\nabla f(x(t)) - \dot x(t) \frac{\beta + 1}{t},
\]
which is of the form of Nesterov's ODE up to the reparameterization $\alpha = \beta + 1$. It is easy to verify that the conditions $\eta \geq \dot r$ and $a = \eta / r$ are verified, thus as a consequence of Corollary~\ref{cor:deterministic_rate}, $f(x(t)) - f(x^\star) = \Ocal(1/r(t)) = \Ocal(1/t^2)$, which is analogous to the quadratic rate of Nesterov's accelerated method in discrete time.

\section{Asymptotic rates for (non-accelerated) $\SMD$}
\label{sec:app:smd_rates}
In order to compare the performance of $\SAMD$ to plain (non-accelerated) $\SMD$, we give a brief discussion of the convergence rates for $\SMD$, and derive, in particular, the optimal rate of the sensitivity parameter given a time-varying $\sigma_*(t)$. The results presented in this section are a straightforward extension of the work of~\cite{mertikopoulos2016convergence} to the case where the bound $\sigma_*(t)$ is time-varying. First, we can prove a bound on the rate of change of the energy, similarly to Lemmas~\ref{lem:lyap_bound} and~\ref{lem:lyap_bound_sto}, which leads to the following bound on expected function values.

\begin{theorem}
Suppose that Assumptions~\ref{assumption:lip} and \ref{assumption:volatility} hold. Let $L_{\MD}(z, t) = s(t)D_{\psi^*}(z/s(t), z^\star)$, and let $(X(t), Z(t))$ denote the solution of $\SMD_s$ with initial conditions $(x_0, z_0)$. Then
\al{
\Ebb[ f(\bar X(t))] - f(x^\star)
&\leq \frac{1}{t - t_0} \parenth{L_{\MD}(z_0, t_0) + \psi(x^\star)s(t) + \frac{n L_{\psi^*}}{2} \int_{t_0}^t \frac{\sigma_*^2(\tau)}{s(\tau)}d\tau}.
}
where $\bar X(t) = \frac{1}{t - t_0} \int_{t_0}^t X(\tau)d\tau$.
\end{theorem}

\begin{proof}
Following the proof of Lemma~\ref{lem:lyap_bound_sto}, we have
\al{
dL_{\MD}(Z, t)
&\leq \braket{\nabla \psi^*(Z/s) - x^\star}{ - \nabla f(X)} + \psi(x^\star)\dot s
+ \braket{V_{\MD}}{dB} + \frac{n L_{\psi^*} \sigma_*^2}{2s} \\
&\leq -(f(X) - f(x^\star)) + \psi(x^\star) \dot s + \braket{V_{\MD}}{dB} + \frac{n L_{\psi^*} \sigma_*^2}{2s}.
}
where we defined $V_{\MD}(t) = \sigma(X(t), t)^T(\nabla \psi^*(Z/s) - \nabla \psi^*(z^\star))$. Rearranging and integrating this bound, we have
\al{
\frac{1}{t - t_0} \int_{t_0}^t &(f(X(\tau)) - f(x^\star))d\tau \\
&\leq \frac{1}{t - t_0} \parenth{L_{\MD}(z_0, t_0) + \psi(x^\star)s(t) + \int_{t_0}^t \braket{V(\tau)}{dB(\tau)} + \frac{n L_{\psi^*}}{2} \int_{t_0}^t \frac{\sigma_*^2(\tau)}{s(\tau)}d\tau}.
}
Taking expectations, the martingale term $\int_{t_0}^t \braket{V_{\MD}(\tau)}{dB(\tau)}$ vanishes, and we conclude using the fact that $f(\bar X(t)) \leq \frac{1}{t-t_0} \int_{t_0}^t f(X(\tau))d\tau$ by Jensen's inequality.
\end{proof}

Given this bound, we can adapt the asymptotic rate of $s(t)$ to the rate of $\sigma_*(t)$. For example, if $\sigma_*(t) = \Ocal(t^{\alpha_\sigma})$ is given, and $s(t) = \Theta(t^{\alpha_s})$, $\alpha_s \geq 0$ is to be chosen, we have
\[
\Ebb[ f(\bar X(t))] - f(x^\star) = \Ocal(t^{\alpha_s-1} + t^{2\alpha_\sigma - \alpha_s}),
\]
and the value of $\alpha_s$ which minimizes the asymptotic rate is $\alpha_s = \max(0, \alpha_\sigma+\frac{1}{2})$ (note that $\alpha_s$ is, by assumption, non-negative since the inverse sensitivity $s$ is by assumption non-decreasing), which results in the following rate
\begin{corollary}
\label{cor:smd_rate}
Suppose that Assumptions~\ref{assumption:lip} and \ref{assumption:volatility} hold. Suppose that $\sigma_*(t) = \Ocal(t^{\alpha_\sigma})$, $\alpha_\sigma < \frac{1}{2}$. Then let $s(t) = t^{\max(0, \alpha_\sigma+\frac{1}{2})}$, and $(X(t), Z(t))$ be the unique solution of $\SMD_{s}$. Then
\[
\Ebb[ f(\bar X(t))] - f(x^\star) = \Ocal(t^{\max(\alpha_\sigma - \frac{1}{2}, -1)}).
\]
\end{corollary}

In particular, $\SMD$ cannot adapt to a noise decay that is faster than $\Ocal(t^{-\frac{1}{2}})$, as opposed to $\SAMD$. In the regime where the noise decay is slower than $t^{-\frac{1}{2}}$, comparing this result to Corollary~\ref{cor:sto_rate}, it may appear that $\SAMD$ and $\SMD$ can both achieve the same rate, $\Ocal(t^{\alpha_\sigma - \frac{1}{2}})$, although in practice, $\SAMD$ exhibits more desirable properties, and appears to be numerically more robust to higher levels of noise, and larger discretization step sizes, as discussed in Section~\ref{sec:numerics}.

\end{document}